\documentclass[a4paper, 11pt]{amsart}
\usepackage{amsmath, amssymb, amsthm, bm}	%ams environment
\usepackage{eucal}
\usepackage{braket, stmaryrd,mathrsfs}
\usepackage{cite}
\usepackage[dvipdfmx]{graphicx,xcolor}

\calclayout

\newcommand{\pr}[1]{#1^{\prime}}

\newcommand{\del}{\partial}

\newcommand{\mfrak}[1]{\mathfrak{#1}}
\newcommand{\mcal}[1]{\mathcal{#1}}
\newcommand{\mbb}[1]{\mathbb{#1}}
\newcommand{\mrm}[1]{\mathrm{#1}}
\newcommand{\msf}[1]{\mathsf{#1}}

\newcommand{\what}[1]{\widehat{#1}}
\newcommand{\no}[1]{:\hspace{-3pt} #1\hspace{-3pt}:\hspace{3pt}}

\makeatletter
\newcommand{\xRightarrow}[2][]{\ext@arrow 0359\Rightarrowfill@{#1}{#2}}
\makeatother

\theoremstyle{plain}
\newtheorem{thm}{Theorem}[section]
\newtheorem{lem}[thm]{Lemma}
\newtheorem{prop}[thm]{Proposition}
\newtheorem{cor}[thm]{Corollary}
\theoremstyle{definition}
\newtheorem{defn}[thm]{Definition}%[section]
\theoremstyle{remark}
\newtheorem{rem}[thm]{Remark}%[section]
\newtheorem{exam}[thm]{Example}
\makeatletter
    
    \@addtoreset{equation}{section}
\makeatother

\title[Multiple backward SLE and GFF]{Multiple backward Schramm--Loewner evolution and coupling with Gaussian free field}
\author{Shinji Koshida}
\address{Department of Physics, Faculty of Science and Engineering, Chuo University, Kasuga, Bunkyo, Tokyo 112-8551, Japan}
\email{koshida@phys.chuo-u.ac.jp}

\begin{document}

\begin{abstract}
It is known that a backward Schramm--Loewner evolution (SLE) is coupled with a free boundary Gaussian free field (GFF)
with boundary perturbation to give conformal welding of quantum surfaces.
Motivated by a generalization of conformal welding for quantum surfaces with multiple marked boundary points, we propose a notion of multiple backward SLE.
To this aim, we investigate the commutation relation between two backward Loewner chains,
and consequently, we find that the driving process of each backward Loewner chain has to have a drift term given by
logarithmic derivative of a partition function, which is determined by a system of Belavin--Polyakov--Zamolodchikov-like equations
so that these Loewner chains are commutative.
After this observation, we define a multiple backward SLE as a tuple of mutually commutative backward Loewner chains.
It immediately follows that each backward Loewner chain in a multiple backward SLE is obtained as a Girsanov transform of a backward SLE.
We also discuss coupling of a multiple backward SLE with a GFF with boundary perturbation 
and find that a partition function and a boundary perturbation are uniquely determined
so that they are coupled with each other.
\end{abstract}


\maketitle

%\tableofcontents

\section{Introduction}
Recent studies on Schramm--Loewner evolution (SLE) coupled with two-dimensional Gaussian free field (GFF) 
\cite{Dubedat2009,SchrammSheffield2009,SchrammSheffield2013,IzyurovKytola2013, DuplantierMillerSheffield2014,Sheffield2016,MillerSheffield2016a,MillerSheffield2016b,MillerSheffield2016c, MillerSheffield2017} have created a new trend in random geometry
leading to a canonical construction of SLE from GFF and an insight into underlying geometry of GFF.
In these studies, a GFF is an ingredient of random objects such as  a {\it quantum surface} \cite{DuplantierMillerSheffield2014, Sheffield2016} or an {\it imaginary surface} \cite{MillerSheffield2016a,MillerSheffield2016b,MillerSheffield2016c, MillerSheffield2017}, roughly,
the former (resp. the latter) of which is an equivalence class of two-dimensional simply connected domains equipped with random metrics
(resp. random vector fields).
Given a quantum surface uniformized to the complex upper half plane, then, one can think of matching boundary segments lying on both sides of the origin
so that they have the same length with respect to the random metric and gluing them together.
Consequently, one obtains a random curve growing in the complex upper half plane and could consider the conformal welding problem
that requires us to determine its probability law.
In the case that an imaginary surface uniformized to the complex upper half plane is given,
one sees a flow line starting at the origin along the random vector field
and could consider the flow line problem that requires us to determine its probability law.
It has been proved \cite{Sheffield2016,MillerSheffield2016a,MillerSheffield2016b,MillerSheffield2016c, MillerSheffield2017} that, for a quantum surface and an imaginary surface with proper boundary perturbations, both problems are solved by SLE
relying on the coupling of SLE with GFF.

Due to the boundary perturbations, the quantum surfaces (resp. the imaginary surfaces) subject to the conformal welding problem (resp. the flow line problem)
can be regarded as being equipped with two marked boundary points (resp. boundary condition changing points) at the origin and infinity.
Therefore, it seems natural to consider analogues of these problems in the case when the quantum surfaces (resp. the imaginary surfaces)
are equipped with more marked boundary points (resp. boundary condition changing points) than two.
In the previous work \cite{KatoriKoshida2020a}, we posed such generalizations and found that they are solved by multiple SLE \cite{BauerBernardKytola2005,Dubedat2006,Dubedat2007,Graham2007, KytolaPeltola2016,PeltolaWu2019},
but we also encountered a new problem.

The couplings of SLE with GFF to solve the conformal welding problem and the flow line problem are slightly different.
While, in the case of the flow line problem, the coupling of the usual forward flow of SLE \cite{Schramm2000, RohdeSchramm2005}
and GFF under proper boundary condition is useful,
in the case of the conformal welding problem, one has to make a backward SLE coupled with a free boundary GFF with a proper boundary perturbation.
These differences not only persist when we move on to the case with multiple marked boundary points/boundary condition changing points,
but also get more serious.
It is known \cite{Lawler2009b} that a forward SLE and a backward SLE are roughly the inverse mapping of each other,
which is why a forward SLE and a backward SLE generate essentially the same random curve.
Note that the proof of this fact relies on the property that, for a Brownian motion $(B_{t}:t\ge 0)$ and a fixed time $T>0$,
the stochastic process $(B_{T-t}-B_{T}:t\in [0,T])$ is again a Brownian motion.
Therefore, for a multiple SLE, whose driving process has a drift term apart from a Brownian motion, the same thing cannot be expected.
Nevertheless, a multiple backward SLE naturally gives a solution to the conformal welding problem for a quantum surface
with multiple marked boundary points.
The new problem mentioned above and that we address in this paper is how a multiple backward SLE makes sense as a stochastic process
generating random curves.

Let us take a quick look at construction of a multiple SLE in forward case
based on the commutation relation between Loewner chains \cite{Dubedat2006,Dubedat2007,Graham2007}.
Suppose that we have two Loewner chains $\left(g_{t}(\cdot):t\ge 0\right)$ and $\left(\tilde{g}_{s}(\cdot):s\ge 0\right)$ driven by some It{\^o} processes.
Using these Loewner chains, one can think of two schemes of generating multiple curves:
One scheme is to generate a curve according to $\left(g_{t}(\cdot):t\ge 0\right)$
and next to generate the other curve in the remaining domain letting $\left(\tilde{g}_{s}(\cdot):s\ge 0\right)$ evolve,
and the other one is to do the same thing in the converted order.
In both schemes, one obtains two random curves in the complex upper half plane.
Then, the requirement that their probability laws are identical
imposes strict conditions on the driving processes of the Loewner chains.
In particular, it can be argued that they share a function that solves a system of Belavin--Polyakov--Zamolodchikov(BPZ)-like equations
so that their drift terms are given by its logarithmic derivatives.
What is called a multiple SLE these days \cite{KytolaPeltola2016,PeltolaWu2019} is a multiple of Loewner chains, the driving process of each of which
has a drift term given by a logarithmic derivative of a single function solving a system of BPZ equations.
Owing to the argument of the commutation relation, it is ensured that these Loewner chains
consistently generate multiple curves in the complex half plane.
It is also known \cite{Werner2004b,SchrammWilson2005,KytolaPeltola2016,PeltolaWu2019} that, for a multiple SLE, each Loewner chain is a Girsanov transform of a usual SLE up to some stopping time.

We also comment that a multiple SLE was also constructed in \cite{BauerBernardKytola2005},
where it was thought of as a Loewner chain generating multiple curves,
which can be regarded as a stochastic version of the multiple slit Loewner theory \cite{RothSchleissinger2017}
and was adopted in our previous works \cite{KatoriKoshida2020a,KatoriKoshida2020b}.
In \cite{BauerBernardKytola2005}, drift terms in driving processes were derived in connection to conformal field theory (CFT)
whose probability theoretical origin was later clarified in \cite{Graham2007}.

Our aim is to carry out an analogous discussion of commutation relation as above for the backward case.
As was expected, 
\begin{description}
\item[Rough statement of Thorem \ref{thm:commutation}]
{\it the commutation relation imposes conditions on the driving processes of the backward Loewner chains under consideration
so that the drift terms are given by logarithmic derivatives of a function that is a solution of a system of BPZ equations},
\end{description}
but parameters in the BPZ equations appear in a different way from the case of a multiple forward SLE.
To define a backward multiple SLE, we turn this argument upside down and start from a solution of a system of BPZ equations,
which we call a partition function.
Then a multiple backward SLE associated with that partition function is defined as a multiple of
backward Loewner chains, whose driving processes have drift terms determined by logarithmic derivatives of the partition function.
Similarly as in the case of a multiple forward SLE, these backward Loewner chains consistently generate multiple random curves.
It can be also seen that
\begin{description}
\item[Rough statement of Theorem \ref{thm:Girsanov_transform}]
{\it each backward Loewner chain is a Girsanov transform of a usual backward SLE
with the Radon--Nikod{\'y}m derivative being written in terms of the partition function.}
\end{description}
Therefore, a multiple backward SLE is equivalently defined as a multiple of probability measures
each of which is a suitable Girsanov transform of the law of an ordinary backward SLE.

After fixing a definition of a multiple backward SLE, we discuss coupling between a multiple backward SLE
and a free boundary GFF with boundary perturbation.
We begin with a precise definition of coupling in such a way that a multiple backward SLE coupled with a free boundary GFF with boundary perturbation
gives a solution to the associated conformal welding problem.
Then, we find that 
\begin{description}
\item[Rough statement of Theorem \ref{thm:coupling_constraint}]
{\it the requirement that a multiple backward SLE is coupled with a free boundary GFF with boundary perturbation
imposes constraints on both the multiple backward SLE and the boundary perturbation
that are strict enough to fix them essentially uniquely.}
\end{description}
We also prove an analogue of Theorem \ref{thm:coupling_constraint} for a multiple forward SLE in Theorem \ref{thm:coupling_forward}.

Let us make some comments on difference and relation between the current work and our previous work \cite{KatoriKoshida2020a}.
In the previous work, we considered a multiple backward SLE that generates multiple curves at once.
On the other hand, what we call a multiple backward SLE in the current work is a consistent family of backward Loewner chains
by which multiple curves are generated one by one.
In the previous work \cite{KatoriKoshida2020a}, we obtained a sufficient condition for a multiple backward SLE that generates multiple curves at once to be coupled with a free boundary GFF (see also \cite{KatoriKoshida2020b}). To be precise, a multiple backward SLE is coupled with a free boundary GFF if the system of driving processes is given by a Dyson model \cite{Dyson1962}. We did not, however, manage to prove the converse direction.
In the present work, we study a different multiple backward SLE, a family of backward Loewner chains, and find in Theorem \ref{thm:coupling_constraint} the necessary and sufficient conditions for the multiple backward SLE to be coupled with a free boundary GFF.
In a subsequent work of ours \cite{KatoriKoshida2020c}, we will prove the converse statement of that in \cite{KatoriKoshida2020a}, and the equivalence between \cite{KatoriKoshida2020a} and the present work as well.

An implication of Theorem \ref{thm:coupling_constraint} seems to be of great importance.
At first, we intended to design a boundary perturbation so that the associated conformal welding problem
is solved by a desired multiple backward SLE,
but, consequently, Theorem \ref{thm:coupling_constraint} prohibited us from carrying out that program except for one case.
Then, a new problem arises whether it is possible to construct other multiple backward SLE by considering a generalization of conformal welding problem
or whether the chosen multiple backward SLE is the only one that can be constructed starting from the theory of GFF.

Before closing this introduction, we briefly comment on future directions.
It would be interesting to consider other kinds of SLE such as a radial SLE, a quadrant SLE \cite{Takebe2014} and an SLE$(\kappa,\rho)$
to generalize Theorems \ref{thm:coupling_constraint} and \ref{thm:coupling_forward}.
We are in particular interested in cases of multiply connected domains
that are treated by means of an annulus SLE \cite{Zhan2004,ByunKangTak2018}
or a stochastic Komatu-Loewner evolution \cite{BauerFriedrich2008, ChenFukushima2018, Murayama2019}.

This paper is organized as follows.
In the next Sect. \ref{sect:commutation_relation}, after fixing our terminologies concerning backward Loewner chains,
we investigate commutation relation between two backward Loewner chains and prove Theorem \ref{thm:commutation}.
We also discuss the mutual commutativity among backward Loewner chains extending the result of Theorem \ref{thm:commutation},
following which, in Sect. \ref{sect:multiple_backward_SLE}, we define a multiple backward SLE as a special case of
a mutually commuting family of backward Loewner chains.
We also prove Theorem \ref{thm:Girsanov_transform} and pose an equivalent definition of a multiple backward SLE
as a multiple of probability measures,
with which we work in Sect. \ref{sect:coupling_GFF}.
In Sect. \ref{sect:coupling_GFF}, we consider coupling of a multiple backward SLE with a free boundary GFF with boundary perturbation.
To this aim, we begin with a review of free boundary GFF and then give a definition of coupling.
We will find that the coupling conditions impose strict constraints on both the multiple backward SLE and the boundary perturbation
to give Theorem \ref{thm:coupling_constraint}.
In this paper, we avoid an explicit use of CFT and carry our discussion in purely a probability theoretical manner.
For readers familiar with CFT, however, it might be more convenient to see CFT background underlying our discussion.
In Appendix \ref{app:CFT}, we summarize how observables that play significant roles in our discussion originate as correlation functions of CFT.
Though we focus on a multiple backward SLE in this paper, 
an analogue of Theorem \ref{thm:coupling_constraint} can also be considered for an ordinary multiple forward SLE.
In Appendix \ref{sect:forward_flow}, we discuss a multiple forward SLE coupled with a Dirichlet boundary GFF with boundary perturbation.
We recommend readers to read Appendix \ref{sect:forward_flow} separately from the main text because,
to avoid notational complexity, we use the same symbols as in the main text with different definitions.

\subsection*{Terminologies}
Let $\mbb{H}=\{z\in\mbb{C}|\mrm{Im}z>0\}$ be the complex upper half-plane
and let $\overline{\mbb{H}}$ be its closure in $\mbb{C}$.
A subset $K\subset \mbb{H}$ is called a compact $\mbb{H}$-hull if $K=\mbb{H}\cap \overline{K}$ and $\mbb{H}\backslash K$ is simply connected.
For a compact $\mbb{H}$-hull $K$, there exists a unique conformal transformation $g_{K}:\mbb{H}\backslash K\to\mbb{H}$ under the hydrodynamical normalization at infinity:
\begin{equation*}
	\lim_{z\to\infty}|g_{K}(z)-z|=0.
\end{equation*}
We define the half-plane capacity of $K$ at infinity by
\begin{equation*}
	\mrm{hcap}(K):=\lim_{z\to\infty}z(g_{K}(z)-z).
\end{equation*}

For $N\in\mbb{N}$, we set
\begin{equation*}
	\mrm{Conf}_{N}(\mbb{R}):=\left\{\bm{x}=(x_{1},\dots, x_{N})\in\mbb{R}^{N}|x_{i}\neq x_{j}\mbox{ if }i\neq j\right\}
\end{equation*}
as the collection of $N$-point configurations on $\mbb{R}$.
Note that this space is the union of $N!$ connected components, and each connected component is simply connected.

\subsection*{Acknowledgements}
The author is grateful to Yoshimichi Ueda and Takuya Murayama for stimulating his interest in the subject of the present paper,
and to Makoto Katori, Makoto Nakashima and Noriyoshi Sakuma for discussions and opportunities to talk in seminars they arranged.
He also thanks the anonymous referee for helping the author dramatically improve the manuscript with useful suggestions.
This work was supported by the Grant-in-Aid for JSPS Fellows (No.~19J01279).

\section{Commutation relation}
\label{sect:commutation_relation}
In this section, we investigate the commutation relation between two backward Loewner chains
and derive conditions so that they consistently generate two curves.
To this aim, we begin with fixing our terminologies concerning backward Loewner chains.

\begin{defn}
Let $U:[0,\infty)\to\mbb{R}$ be a continuous function.
The backward Loewner chain $\left(f_{t}(\cdot):t\ge 0 \right)$ driven by $U$ is the solution of the equation
\begin{equation*}
	\frac{d}{dt}f_{t}(z)=-\frac{2}{f_{t}(z)-U(t)},\quad t\ge 0,\quad f_{0}(z)=z \in \mbb{H}.
\end{equation*}
\end{defn}

For a backward Loewner chain $(f_{t}(\cdot):t\ge 0)$ driven by a continuous function $U$ and a fixed point $z\in \mbb{H}$,
the real-valued functions $x_{t}(z):=\mrm{Re}f_{t}(z)$, $y_{t}(z):=\mrm{Im}f_{t}(z)$, $t\geq 0$ satisfy the system of ordinary differential equations (ODEs)
\begin{align*}
	\frac{d}{dt}x_{t}(z)=-\frac{2(x_{t}(z)-U(t))}{(x_{t}(z)-U(t))^{2}+y_{t}(z)^{2}}, \quad
	\frac{d}{dt}y_{t}(z)=\frac{2y_{t}(z)}{(x_{t}(z)-U(t))^{2}+y_{t}(z)^{2}}, \quad t\geq 0,
\end{align*}
under the initial conditions $x_{0}(z)=\mrm{Re}z\in\mbb{R}$, $y_{0}(z)=\mrm{Im}z>0$.
This implies, due to the general theory of ODEs, that, at each $t\geq 0$, $f_{t}(\mbb{H})$ lies in $\mbb{H}$ and $K_{t}:=\mbb{H}\backslash f_{t}(\mbb{H})$ is a compact $\mbb{H}$-hull.
When we set $h_{t}:=f_{t}^{-1}$, $t\geq 0$, we can see that $(h_{t}(\cdot):t\geq 0)$ satisfies the partial differential equation
\begin{equation}
\label{eq:Loewner_Kufarev_backward}
	\frac{\del}{\del t}h_{t}(w)=\frac{2}{w-U(t)}\frac{\del}{\del w}h_{t}(w), \quad t\geq 0, \quad h_{0}(w)=w\in\mbb{H},
\end{equation}
and, for each $t\geq 0$, $h_{t}:\mbb{H}\backslash K_{t}\to\mbb{H}$ is a conformal transformation. Note that the domain of definition of $h_{t}$ depends on $t\geq 0$.
Expanding both sides of (\ref{eq:Loewner_Kufarev_backward}) around infinity, we can see that, for each $t\geq 0$, $h_{t}$ is hydrodynamically normalized and that $\mrm{hcap}(K_{t})=2t$, $t\geq 0$.

The definition of a backward Loewner chain obviously works even if a continuous function $U$
is replaced by a stochastic process as long as its paths are almost surely continuous.
A fundamental example is the backward SLE$(\kappa)$ defined as follows:
\begin{defn}
Let $\kappa>0$ be fixed.
A backward SLE$(\kappa)$ is the backward Loewner chain $(f_{t}(\cdot):t\ge 0)$ driven by $(W_{t}=\sqrt{\kappa}B_{t}:t\ge 0)$
where $(B_{t}:t\ge 0)$ is a standard Brownian motion.
\end{defn}
It has been known \cite{RohdeSchramm2005,Kang2007,Lind2008, Lawler2009b}
that a backward SLE is easier to analyze in many ways than a forward SLE.
More recent studies on backward SLE include \cite{RohdeZhan2016,MackeyZhan2019}.
A backward SLE$(\kappa)$ is roughly the {\it inverse mapping} of an SLE$(\kappa)$.
A proof of the following fact can be found e.g. in \cite{Lawler2009b}.
\begin{prop}
\label{prop:backward_SLE_gives_inverse_mapping}
Let $\kappa>0$ and let $(f_{t}(\cdot):t\ge 0)$ be a backward SLE$(\kappa)$ driven by $(W_{t}:t\ge 0)$.
Also let $(g_{t}(\cdot):t\ge 0)$ be an SLE($\kappa$), {\it i.e.}, it is the solution of
\begin{equation*}
	\frac{d}{dt}g_{t}(z)=\frac{2}{g_{t}(z)-\widetilde{W}_{t}},\quad t\ge 0,\quad g_{0}(z)=z\in\mbb{H},
\end{equation*}
where we put $\widetilde{W}_{t}=\sqrt{\kappa}\widetilde{B}_{t}$, $t\ge 0$ with $(\widetilde{B}_{t}:t\ge 0)$ being a standard Brownian motion.
We set $\hat{f}_{t}(z):=f_{t}(z)-W_{t}$, $t\ge 0$ and $\hat{g}_{t}(z):=g_{t}(z)-\widetilde{W}_{t}$, $t\ge 0$.
Then, at each $t>0$, we have
\begin{equation*}
	\hat{f}_{t}(\cdot)\overset{(\mrm{law})}{=}\hat{g}_{t}^{-1}(\cdot).
\end{equation*}
\end{prop}
Recall that, for an SLE$(\kappa)$ $(g_{t}(\cdot):t\geq 0)$, there is a random curve $\eta:[0,\infty)\to\overline{\mathbb{H}}$ and, at each $t\geq 0$, $g_{t}$ is the hydrodynamically normalized conformal transformation from the unbounded component of $\mbb{H}\backslash \eta(0,t]$ to $\mbb{H}$.
Let $(K_{t}:t\geq 0)$ be the family of compact $\mbb{H}$-hulls generated by a backward SLE$(\kappa)$ $(f_{t}(\cdot):t\geq 0)$ driven by $(W_{t}:t\geq 0)$.
Then, Proposition \ref{prop:backward_SLE_gives_inverse_mapping} implies that, at each $t\geq 0$, the probability law of $\mbb{H}\backslash (K_{t}-W_{t})$ coincides with that of the unbounded component of $\mbb{H}\backslash \eta(0,t]$.
In particular, if $\kappa\in (0,4]$, $K_{t}$ at each $t\geq 0$ is a simple curve a.s. since so is $\eta (0,t]$.
It is not, however, true that there exists a simple curve $\widetilde{\eta}:[0,\infty)\to\mbb{H}$ such that $K_{t}=\widetilde{\eta}(0,t]$, $t\geq 0$.

The proof of Proposition \ref{prop:backward_SLE_gives_inverse_mapping} relies on the fact that for a Brownian motion $(B_{t}:t\ge 0)$ and $T>0$,
the stochastic process $(B_{T-t}-B_{T}:t\in [0,T])$ is again a Brownian motion.
Therefore, we cannot expect the same property for a backward Loewner chain driven by a stochastic process with a drift term.

\begin{defn}
Let $\kappa>0$, $N\in\mbb{N}$ and $i\in \{1,\dots, N\}$ be fixed and 
let $b=b(x_{1},\dots, x_{N})$ be a function on $\mrm{Conf}_{N}(\mbb{R})$
that is translation invariant and homogeneous of degree $-1$.
We consider the stochastic process $\left(\bm{X}^{(i)}_{t}=(X^{(i,1)}_{t},\dots, X^{(i,N)}_{t}):t\geq 0\right)$ satisfying
\begin{align*}
	dX^{(i,i)}_{t}&=\sqrt{\kappa}dB_{t}+b(\bm{X}_{t}^{(i)})dt,\quad t\ge 0, \\
	\frac{d}{dt}X^{(i,j)}_{t}&=-\frac{2}{X_{t}^{(i,j)}-X^{(i,i)}_{t}},\quad t\ge 0,\quad j\neq i,
\end{align*}
where $(B_{t}:t \ge 0)$ is a standard Brownian motion.
We call the backward Loewner chain $\left(f_{t}(\cdot):t\ge 0 \right)$ driven by the $i$-th component $(X_{t}^{(i,i)}:t\geq 0)$ of the above stochastic process
the $i$-th backward SLE$(\kappa,b)$ driven by the stochastic process $\left(\bm{X}^{(i)}_{t}:t\ge 0\right)$.
For an $N$-point configuration $\bm{X}=(X_{1},\dots, X_{N})\in\mrm{Conf}_{N}(\mbb{R})$,
we say that the $i$-th backward SLE$(\kappa,b)$ starts at $\bm{X}$
if $\bm{X}^{(i)}_{0}=\bm{X}$.
\end{defn}

\begin{rem}
One must not be confused in usage of the term ``driving process".
For an $i$-th SLE$(\kappa,b)$ driven by $(\bm{X}^{(i)}_{t}:t\ge 0)$, only the $i$-th process $(X^{(i,i)}_{t}:t\ge 0)$ plays the role of the driving process of a Loewner chain.
It is, however, convenient to call $(\bm{X}^{(i)}_{t}:t\ge 0)$ the driving process of the $i$-th SLE$(\kappa,b)$ in the case when one needs to keep track of other points as well.
\end{rem}

The assumption that the function $b$ is translation invariant and homogeneous of degree $-1$
ensures that the law of the associated family of compact $\mbb{H}$-hulls $(K_{t}:t\geq 0)$ is conformally invariant.
Indeed, this homogeneity of $b$ gives the property that
\begin{equation*}
	d\left(\lambda X^{(i,i)}_{t}\right)=\sqrt{\kappa}dB_{\lambda^{2}t}+b(\lambda\bm{X}^{(i)}_{t})d(\lambda^{2}t),\ \ t\ge 0,
\end{equation*}
for an arbitrary constant $\lambda>0$.

Suppose that $\kappa_{i}>0$, $i=1,\dots, N$ and functions $b_{i}$, $i=1,\dots, N$ on $\mrm{Conf}_{N}(\mbb{R})$ that are translation invariant and homogeneous of degree $-1$ are given.
For each $i=1,\dots, N$, let $\left(f^{(i)}_{t}(\cdot):t\ge 0\right)$ be an $i$-th backward SLE$(\kappa_{i},b_{i})$ driven by a stochastic process $\left(\bm{X}^{(i)}_{t}\in\mrm{Conf}_{N}(\mbb{R}):t\ge 0\right)$.
We write the filtration associated with $\left(\bm{X}^{(i)}_{t}:t\ge 0\right)$ as $(\mcal{F}^{(i)}_{t})_{t\ge 0}$, $i=1,\dots, N$,
and assume that $\left(\mcal{F}^{(i)}_{t}\right)_{t\ge 0}$ are mutually independent.

Let us fix a pair $\{i,j\}\subset\{1,\dots,N\}$.
Using the backward Loewner chains introduced above,
we have two schemes of generating two compact $\mbb{H}$-hulls given an $N$-point configuration $\bm{X}=(X_{1},\dots, X_{N})\in\mrm{Conf}_{N}(\mbb{R})$ (see Figure \ref{fig:commutation_relation}):
\begin{description}
\item[Scheme 1]	Generate a compact $\mbb{H}$-hull $K^{(i)}_{\varepsilon}$ according to the $i$-th backward SLE$(\kappa_{i},b_{i})$ $(f^{(i)}_{t}(\cdot):t\ge 0)$
				starting at $\bm{X}$ up to a time $\varepsilon >0$.
				Next, forgetting the first compact $\mbb{H}$-hull $K^{(i)}_{\varepsilon}$, generate the second compact $\mbb{H}$-hull $K^{(j)}_{\tilde{\varepsilon}}$
				letting the $j$-th backward SLE$(\kappa_{j},b_{j})$ $\left(f^{(j)}_{s}(\cdot):s\ge 0\right)$ starting at $\bm{X}^{(i)}_{\varepsilon}$
				evolve up to a time $\tilde{\varepsilon}>0$.
				We also require that
				\begin{equation*}
					\mrm{hcap}\left(f^{(j)}_{\tilde{\varepsilon}}(K^{(i)}_{\varepsilon})\right)=2c\tilde{\varepsilon}
				\end{equation*}
				for a fixed $c>0$.
				Then, one obtains the union of two compact $\mbb{H}$-hulls $K^{1}_{c,\tilde{\varepsilon}}:=f^{(j)}_{\tilde{\varepsilon}}(K^{(i)}_{\varepsilon})\cup K^{(j)}_{\tilde{\varepsilon}}$.
\item[Scheme 2] 	Generate a compact $\mbb{H}$-hull $K^{(j)}_{\pr{\varepsilon}}$ according to the $j$-th backward SLE$(\kappa_{j},b_{j})$ $\left(f^{(j)}_{s}(\cdot):s\ge 0\right)$ starting at
				$\bm{X}$ up to a time $\pr{\varepsilon}>0$.
				Next, forgetting the first compact $\mbb{H}$-hull $K^{(j)}_{\pr{\varepsilon}}$, generate the second one $K^{(i)}_{c\tilde{\varepsilon}}$
				letting the $i$-th backward SLE$(\kappa_{i},b_{i})$ $\left(f^{(i)}_{t}(\cdot):t\ge 0\right)$ starting at $\bm{X}^{(j)}_{\pr{\varepsilon}}$
				evolve up to a time $c\tilde{\varepsilon}>0$, where $\tilde{\varepsilon}>0$ and $c>0$ are those taken in {\bf Scheme 1}.
				We also require that
				\begin{align*}
					\mrm{hcap}\left(f^{(i)}_{c\tilde{\varepsilon}}(K^{(j)}_{\pr{\varepsilon}})\right)=2\tilde{\varepsilon}.
				\end{align*}
				Then, one obtains the union of two compact $\mbb{H}$-hulls $K^{2}_{c,\tilde{\varepsilon}}:=K^{(i)}_{c\tilde{\varepsilon}}\cup f^{(i)}_{c\tilde{\varepsilon}}(K^{(j)}_{\pr{\varepsilon}})$.
\end{description}
Notice that $\varepsilon$ and $\pr{\varepsilon}$ are determined by $\tilde{\varepsilon}$ and $c$. In the subsequent arguments, we think of $\tilde{\varepsilon}$ and $c$ as independent parameters.

\begin{defn}
The $i$-th backward SLE$(\kappa_{i},b_{i})$ and $j$-th backward SLE$(\kappa_{j},b_{j})$ are said to be commutative if
$K^{1}_{c,\tilde{\varepsilon}}\overset{(\mrm{law})}{=}K^{2}_{c,\tilde{\varepsilon}}$ for an arbitrary initial condition $\bm{X}\in\mrm{Conf}_{N}(\mbb{R})$ and arbitrary $\tilde{\varepsilon}>0$, $c>0$.
\end{defn}

\begin{figure}
\begin{center}
\includegraphics[width=\linewidth]{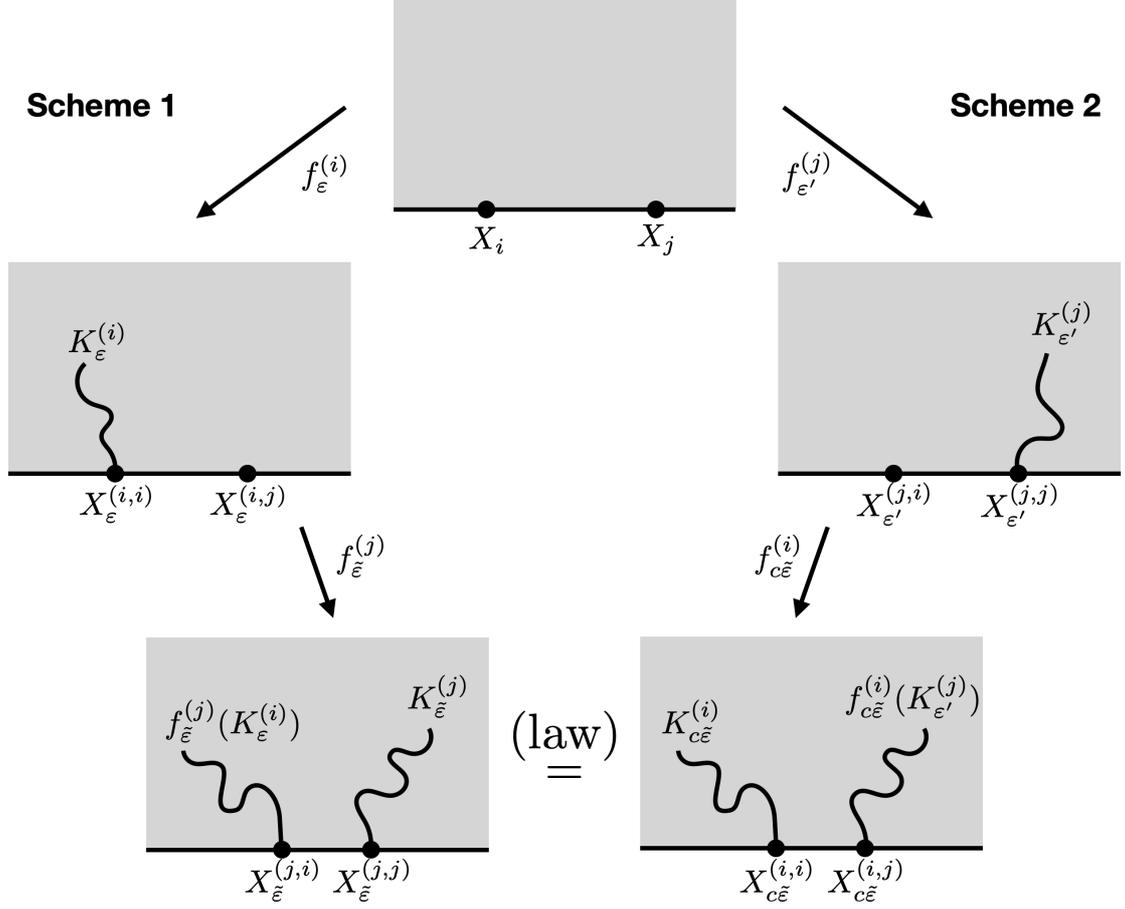}
\caption{Commutation relation of backward Loewner chains.}
\label{fig:commutation_relation}
\end{center}
\end{figure}

Here is the main theorem in this section.
\begin{thm}
\label{thm:commutation}
The $i$-th backward SLE$(\kappa_{i},b_{i})$ and $j$-th backward SLE$(\kappa_{j},b_{j})$ are commutative if and only if
the following conditions are satisfied:
\begin{enumerate}
\item 	Either $\kappa_{i}=\kappa_{j}$ or $\kappa_{i}=16/\kappa_{j}$.
\item 	There exists a translation invariant homogeneous function $\mcal{Z}=\mcal{Z}(x_{1},\dots, x_{N})$ on $\mrm{Conf}_{N}(\mbb{R})$
		with the following properties:
	\begin{enumerate}
	\item 	The functions $b_{k}$, $k=i,j$ are given by $b_{k}=\kappa_{k}\del_{x_{k}}\log\mcal{Z}$, $k=i,j$.
	\item 	There exists a function $F_{ij}=F_{ij}(x,x_{1},\dots,\hat{x}_{i},\dots,\hat{x}_{j},\dots, x_{N})$ on $\mrm{Conf}_{N-1}(\mbb{R})$
			homogeneous of degree $-2$ such that
			\begin{align*}
				\left(\frac{\kappa_{i}}{2}\del_{x_{i}}^{2}-\sum_{k;k\neq i}\frac{2}{x_{k}-x_{i}}\del_{x_{k}}+\frac{2h_{\kappa_{j}}}{(x_{j}-x_{i})^{2}}+F_{ij}(x_{i},\bm{x})\right)\mcal{Z}&=0, \\
				\left(\frac{\kappa_{j}}{2}\del_{x_{j}}^{2}-\sum_{k;k\neq j}\frac{2}{x_{k}-x_{j}}\del_{x_{k}}+\frac{2h_{\kappa_{i}}}{(x_{i}-x_{j})^{2}}+F_{ij}(x_{j},\bm{x})\right)\mcal{Z}&=0,
			\end{align*}
			where we set $h_{\kappa}=-\frac{\kappa+6}{2\kappa}$.
	\end{enumerate}
	Moreover, the function $\mcal{Z}$ is unique up to multiplicative constant.
\end{enumerate}	
\end{thm}

\begin{proof}
The stochastic processes $\left(\bm{X}^{(k)}_{t}:t\ge 0\right)$, $k=i,j$ are Markov processes.
Thinking of $\mrm{Conf}_{N}(\mbb{R})$ as a subset of $\mbb{R}^{N}$,
their generators are derived by means of It{\^o}'s formula so that
\begin{equation*}
	\mcal{L}_{k}=\frac{\kappa_{k}}{2}\del_{x_{k}}^{2}+b_{k}(\bm{x})\del_{x_{k}}-\sum_{\ell;\ell \neq k}\frac{2}{x_{\ell}-x_{k}}\del_{x_{\ell}}, \quad k=i,j.
\end{equation*}

First let us determine the time $\varepsilon$ in {\bf Scheme 1} in terms of $\tilde{\varepsilon}$.
Let $\left(f^{(j)}_{s}(\cdot):s\ge 0\right)$ be the $j$-th backward SLE$(\kappa_{j},b_{j})$ starting at $\bm{X}^{(i)}_{\varepsilon}$.
From the Loewner equation, we have
\begin{equation*}
	\frac{d}{ds}(f^{(j)}_{s})^{\prime}(z)=\frac{2(f^{(j)}_{s})^{\prime}(z)}{(f^{(j)}_{s}(z)-X^{(j,j)}_{s})^{2}},\quad s\ge 0.
\end{equation*}
Then, up to the first order of $\tilde{\varepsilon}$,
\begin{equation*}
	(f^{(j)}_{\tilde{\varepsilon}})^{\prime}(X^{(i,i)}_{\varepsilon})=1+\frac{2\tilde{\varepsilon}}{(X^{(i,i)}_{\varepsilon}-X^{(i,j)}_{\varepsilon})^{2}}+o(\tilde{\varepsilon}).
\end{equation*}
Because of the scaling property of the half-plane capacity, we see that
\begin{equation}
\label{eq:hcap_expansion}
	\mrm{hcap}\left(f^{(j)}_{\tilde{\varepsilon}}(K^{(i)}_{\varepsilon})\right)
	=\left(1+\frac{4\tilde{\varepsilon}}{(X^{(i,i)}_{\varepsilon}-X^{(i,j)}_{\varepsilon})^{2}}\right)2\varepsilon +o(\tilde{\varepsilon}^{2}).
\end{equation}
Note that $o(\tilde{\varepsilon}^{2})$ in (\ref{eq:hcap_expansion}) is independent of $\varepsilon$.
Hence, we can determine $\varepsilon$ in terms of $\tilde{\varepsilon}$ and $c$ up to the second order of $\tilde{\varepsilon}$ by equating (\ref{eq:hcap_expansion}) to $2c\tilde{\varepsilon}$ so that
\begin{equation}
\label{eq:time_scheme1}
	\varepsilon=\left(1-\frac{4\tilde{\varepsilon}}{(X_{i}-X_{j})^{2}}\right)c\tilde{\varepsilon}+o(\tilde{\varepsilon}^{2}).
\end{equation}
Here, $X_{i}$ and $X_{j}$ are the $i$-th and $j$-th components of an initial condition $\bm{X}=(X_{1},\dots, X_{N})\in\mrm{Conf}_{N}(\mbb{R})$, respectively.
In a similar manner, the time $\pr{\varepsilon}$ in {\bf Scheme 2} is determined as
\begin{equation}
\label{eq:time_scheme2}
	\pr{\varepsilon}=\left(1-\frac{4c\tilde{\varepsilon}}{(X_{j}-X_{i})^{2}}\right)\tilde{\varepsilon}+o(\tilde{\varepsilon}^{2}).
\end{equation}

Set $\mcal{G}_{t,s}:=\mcal{F}^{(i)}_{t}\vee \mcal{F}^{(j)}_{s}$, $t,s\ge 0$.
Then, $(\mcal{G}_{t,s})_{t,s\ge 0}$ forms a double filtration of $\sigma$-algebras.
Let $\varphi=\varphi(\bm{x})\in C_{\mrm{b}}^{\infty}(\mbb{R}^{N})$ be a bounded smooth function.
In {\bf Scheme 1}, we see that
\begin{align*}
	\mbb{E}\left[\varphi(\bm{X}^{(j)}_{\tilde{\varepsilon}})\right]
	=\mbb{E}\left[\mbb{E}\left[\varphi(\bm{X}^{(j)}_{\tilde{\varepsilon}})\Big|\mcal{G}_{\varepsilon,0}\right]\right]
	=\mbb{E}\left[\left(e^{\tilde{\varepsilon}\mcal{L}_{j}}\varphi\right)(\bm{X}^{(i)}_{\varepsilon})\right]
	=\left(e^{\varepsilon\mcal{L}_{i}}e^{\tilde{\varepsilon}\mcal{L}_{j}}\varphi\right)(\bm{X}).
\end{align*}
On the other hand, in {\bf Scheme 2}, we have
\begin{align*}
	\mbb{E}\Bigl[\varphi (\bm{X}^{(i)}_{c\tilde{\varepsilon}})\Bigr]
	=\mbb{E}\left[\mbb{E}\left[\varphi (\bm{X}^{(i)}_{c\tilde{\varepsilon}})\Big|\mcal{G}_{0,\pr{\varepsilon}}\right]\right]
	=\mbb{E}\left[\left(e^{c\tilde{\varepsilon}\mcal{L}_{i}}\varphi\right)(\bm{X}^{(j)}_{\pr{\varepsilon}})\right]
	=\left(e^{\pr{\varepsilon}\mcal{L}_{j}}e^{c\tilde{\varepsilon} \mcal{L}_{i}}\varphi\right)(\bm{X}).
\end{align*}
Therefore, the desired equivalence $K^{1}_{c,\tilde{\varepsilon}}\overset{(\mrm{law})}{=}K^{2}_{c,\tilde{\varepsilon}}$ holds if and only if
the following relation among operators is valid:
\begin{equation}
\label{eq:commutation_finite_time}
	e^{\varepsilon\mcal{L}_{i}}e^{\tilde{\varepsilon}\mcal{L}_{j}}=e^{\pr{\varepsilon}\mcal{L}_{j}}e^{c\tilde{\varepsilon} \mcal{L}_{i}}.
\end{equation}
Using the expressions (\ref{eq:time_scheme1}) and (\ref{eq:time_scheme2}), each side becomes
\begin{align*}
	e^{\varepsilon\mcal{L}_{i}}e^{\tilde{\varepsilon}\mcal{L}_{j}}
	&=1+\tilde{\varepsilon}\left(c\mcal{L}_{i}+\mcal{L}_{j}\right)
		+\tilde{\varepsilon}^{2}\left(-\frac{4c\mcal{L}_{i}}{(x_{i}-x_{j})^{2}}+\frac{c^{2}\mcal{L}_{i}^{2}}{2}+c\mcal{L}_{i}\mcal{L}_{j}
				+\frac{\mcal{L}_{j}^{2}}{2}\right) + o(\tilde{\varepsilon}^{2}),\\
	e^{\pr{\varepsilon}\mcal{L}_{j}}e^{c\tilde{\varepsilon}\mcal{L}_{i}}
	&=1+\tilde{\varepsilon}\left(\mcal{L}_{j}+c\mcal{L}_{i}\right)
		+\tilde{\varepsilon}^{2}\left(-\frac{4c\mcal{L}_{j}}{(x_{j}-x_{i})^{2}}+\frac{\mcal{L}_{j}^{2}}{2}+c\mcal{L}_{j}\mcal{L}_{i}
				+\frac{c^{2}\mcal{L}_{i}^{2}}{2}\right) + o(\tilde{\varepsilon}^{2}).
\end{align*}
Therefore, we can see, by comparing the coefficients of $\tilde{\varepsilon}^{2}$, that if the relation (\ref{eq:commutation_finite_time}) holds,
then it follows that the commutation relation between infinitesimal generators
\begin{equation}
\label{eq:commutation_infinitesimal}
	\left[\mcal{L}_{i},\mcal{L}_{j}\right]=\frac{4}{(x_{i}-x_{j})^{2}}\left(\mcal{L}_{i}-\mcal{L}_{j}\right)
\end{equation}
holds.
Note that the commutation relation (\ref{eq:commutation_infinitesimal}) imposes conditions on input data $\kappa_{k}$ and $b_{k}$, $k=i,j$.

Conversely, the analogous argument as in \cite[Section 6]{Dubedat2007} allows us to see that the infinitesimal commutation relation (\ref{eq:commutation_infinitesimal}) ensures the finite commutation relation (\ref{eq:commutation_finite_time}).
Notice that (\ref{eq:commutation_infinitesimal}) implies the finite commutation modulo $o(\tilde{\varepsilon}^{2})$.
Informally speaking, for any $M\in\mbb{N}$, we divide the left hand side of (\ref{eq:commutation_finite_time}) into
\begin{equation*}
	e^{\varepsilon\mcal{L}_{i}}e^{\tilde{\varepsilon}\mcal{L}_{j}}=\left(e^{\frac{\varepsilon}{M}\mcal{L}_{i}}\right)^{M}\left(e^{\frac{\tilde{\varepsilon}}{M}\mcal{L}_{j}}\right)^{M}.
\end{equation*}
The idea is to compare this to the right hand side of (\ref{eq:commutation_finite_time}) by permuting every pair of small pieces $e^{\frac{\varepsilon}{M}\mcal{L}_{i}}$ and $e^{\frac{\tilde{\varepsilon}}{M}\mcal{L}_{j}}$. Here, notice that permuting a pair of these small pieces gives an error term of order $o((\tilde{\varepsilon}/M)^{2})$, and the number of permutations required is of order $M^{2}$.
Consequently, the difference between both sides of (\ref{eq:commutation_finite_time}) consists of roughly $M^{2}$ error terms of $o((\tilde{\varepsilon}/M)^{2})$.
Since $M$ is arbitrary, we can take the limit $M\to\infty$ to see that the accumulation of the error terms vanishes and that (\ref{eq:commutation_finite_time}) holds.

After some computation, we have
\begin{align*}
	&\left[\mcal{L}_{i},\mcal{L}_{j}\right]-\frac{4}{(x_{i}-x_{j})^{2}}\left(\mcal{L}_{i}-\mcal{L}_{j}\right)\\
	&=\left(\kappa_{i}\del_{x_{i}}b_{j}-\kappa_{j}\del_{x_{j}}b_{i}\right)\del_{x_{i}}\del_{x_{j}} \\
	&\hspace{15pt}-\left(\frac{\kappa_{j}}{2}\del_{x_{j}}^{2}b_{i}+b_{j}\del_{x_{j}}b_{i}+\frac{2b_{i}}{(x_{i}-x_{j})^{2}}-\sum_{k;k\neq j}\frac{2\del_{x_{k}}b_{i}}{x_{k}-x_{j}}+\frac{2\kappa_{i}+12}{(x_{i}-x_{j})^{3}}\right)\del_{x_{i}} \\
	&\hspace{15pt}+\left(\frac{\kappa_{i}}{2}\del_{x_{i}}^{2}b_{j}+b_{i}\del_{x_{i}}b_{j}+\frac{2b_{j}}{(x_{j}-x_{i})^{2}}-\sum_{k;k\neq i}^{N}\frac{2\del_{x_{k}}b_{j}}{x_{k}-x_{i}}+\frac{2\kappa_{j}+12}{(x_{j}-x_{i})^{3}}\right)\del_{x_{j}}.
\end{align*}
Therefore, the commutation relation (\ref{eq:commutation_infinitesimal}) is equivalent to the following conditions
\begin{align}
\label{eq:commutation_rot}
	&\kappa_{i}\del_{x_{i}}b_{j}-\kappa_{j}\del_{x_{j}}b_{i}=0,\\
\label{eq:commutation_x_i}
	&\frac{\kappa_{j}}{2}\del_{x_{j}}^{2}b_{i}+b_{j}\del_{x_{j}}b_{i}+\frac{2b_{i}}{(x_{i}-x_{j})^{2}}-\sum_{k;k\neq j}\frac{2\del_{x_{k}}b_{i}}{x_{k}-x_{j}}+\frac{2\kappa_{i}+12}{(x_{i}-x_{j})^{3}}=0,\\
\label{eq:commutation_x_j}
	&\frac{\kappa_{i}}{2}\del_{x_{i}}^{2}b_{j}+b_{i}\del_{x_{i}}b_{j}+\frac{2b_{j}}{(x_{j}-x_{i})^{2}}-\sum_{k;k\neq i}^{N}\frac{2\del_{x_{k}}b_{j}}{x_{k}-x_{i}}+\frac{2\kappa_{j}+12}{(x_{j}-x_{i})^{3}}=0.
\end{align}
Since every connected component of $\mrm{Conf}_{N}(\mbb{R})$ is simply connected,
from (\ref{eq:commutation_rot}), we see that there exists a function $\mcal{Z}=\mcal{Z}(\bm{x})$ on $\mrm{Conf}_{N}(\mbb{R})$ such that
$b_{k}=\kappa_{k}\del_{x_{k}}\log \mcal{Z}$, $k=i,j$.
Note that the function $\mcal{Z}$ is unique up to multiplication by functions independent of $x_{i}$ and $x_{j}$.
Besides, since $b_{k}$, $k=i,j$ are translation invariant and homogeneous of degree $-1$, the function $\mcal{Z}$ is also translation invariant and homogeneous.
Substituting them into (\ref{eq:commutation_x_i}) and (\ref{eq:commutation_x_j}), we see that
\begin{align*}
	\kappa_{i}\del_{x_{i}}\left(\frac{\kappa_{j}}{2}\frac{\del_{x_{j}}^{2}\mcal{Z}}{\mcal{Z}}-\sum_{k;k\neq j}\frac{2}{x_{k}-x_{j}}\frac{\del_{x_{k}}\mcal{Z}}{\mcal{Z}}+\frac{2h_{\kappa_{i}}}{(x_{i}-x_{j})^{2}}\right)&=0,\\
	\kappa_{j}\del_{x_{j}}\left(\frac{\kappa_{i}}{2}\frac{\del_{x_{i}}^{2}\mcal{Z}}{\mcal{Z}}-\sum_{k;k\neq i}\frac{2}{x_{k}-x_{i}}\frac{\del_{x_{k}}\mcal{Z}}{\mcal{Z}}+\frac{2h_{\kappa_{j}}}{(x_{j}-x_{i})^{2}}\right)&=0,
\end{align*}
where we set
\begin{equation*}
	h_{\kappa}=-\frac{\kappa+6}{2\kappa}.
\end{equation*}
Therefore, there exist functions $F_{k}=F_{k}(x, x_{1},\dots,\hat{x}_{i},\dots, \hat{x}_{j},\dots, x_{N})$, $k=i,j$
such that the function $\mcal{Z}$ satisfies the following set of differential equations:
\begin{align*}
	\left(\frac{\kappa_{i}}{2}\del_{x_{i}}^{2}-\sum_{k;k\neq i}\frac{2}{x_{k}-x_{i}}\del_{x_{k}}+\frac{2h_{\kappa_{j}}}{(x_{j}-x_{i})^{2}}+F_{i}(x_{i},\bm{x})\right)\mcal{Z}&=0, \\
	\left(\frac{\kappa_{j}}{2}\del_{x_{j}}^{2}-\sum_{k;k\neq j}\frac{2}{x_{k}-x_{j}}\del_{x_{k}}+\frac{2h_{\kappa_{i}}}{(x_{i}-x_{j})^{2}}+F_{j}(x_{j},\bm{x})\right)\mcal{Z} &=0.
\end{align*}
For this system of differential equations to have a nonzero solution $\mcal{Z}$, the functions $F_{i}$ and $F_{j}$ have to be chosen properly.
To find conditions on $F_{i}$ and $F_{j}$, we set
\begin{align*}
	\mcal{Q}_{i}&=\frac{\kappa_{i}}{2}\del_{x_{i}}^{2}-\sum_{k;k\neq i}\frac{2}{x_{k}-x_{i}}\del_{x_{k}}+\frac{2h_{\kappa_{j}}}{(x_{j}-x_{i})^{2}}, \\
	\mcal{Q}_{j}&=\frac{\kappa_{j}}{2}\del_{x_{j}}^{2}-\sum_{k;k\neq j}\frac{2}{x_{k}-x_{j}}\del_{x_{k}}+\frac{2h_{\kappa_{i}}}{(x_{i}-x_{j})^{2}}.
\end{align*}
Then, $\mcal{Z}$ is annihilated by any operators from the ideal generated by $\mcal{Q}_{i}+F_{i}(x_{i},\bm{x})$ and $\mcal{Q}_{j}+F_{j}(x_{j},\bm{x})$
in the ring of differential operators.
In particular, it is annihilated by the following operator:
\begin{align*}
	&\bigl[\mcal{Q}_{i}+F_{i}(x_{i},\bm{x}),\mcal{Q}_{j}+F_{j}(x_{j},\bm{x})\bigr]-\frac{4}{(x_{i}-x_{j})^{2}}\bigl((\mcal{Q}_{i}+F_{i}(x_{i},\bm{x}))-(\mcal{Q}_{j}+F_{j}(x_{j},\bm{x}))\bigr)\\
	&=\frac{-3(\kappa_{i}-\kappa_{j})(\kappa_{i}\kappa_{j}-16)}{\kappa_{i}\kappa_{j}(x_{i}-x_{j})^{4}}-\frac{4(F_{i}(x_{i},\bm{x})-F_{j}(x_{j},\bm{x}))}{(x_{i}-x_{j})^{2}} \\
	&\hspace{20pt}+\sum_{k;k\neq j}\frac{2\del_{x_{k}}F_{i}(x_{i},\bm{x})}{x_{k}-x_{j}}-\sum_{k;k\neq i}\frac{2\del_{x_{k}}F_{j}(x_{j},\bm{x})}{x_{k}-x_{i}},
\end{align*}
which is just a multiplication operator.
Therefore, if there exists a nonzero solution $\mcal{Z}$, then this operator has to be zero as a function.
The contribution from the fourth order pole of $(x_{i}-x_{j})$ requires that either $\kappa_{i}=\kappa_{j}$ or $\kappa_{i}\kappa_{j}=16$ holds.
Similarly, for the contribution from the second order pole of $(x_{i}-x_{j})$ vanishes, we have to have $\lim_{x_{i}\to x}F_{i}(x_{i},\bm{x})=\lim_{x_{j}\to x}F_{j}(x_{j},\bm{x})$ for any $x$. In other words, there exists a function $F_{ij}=F_{ij}(x,x_{1},\dots,\hat{x}_{i},\dots,\hat{x}_{j},\dots, x_{N})$
such that $F_{i}(x_{i},\bm{x})=F_{ij}(x_{i},\bm{x})$ and $F_{j}(x_{j},\bm{x})=F_{ij}(x_{j},\bm{x})$.
It is also obvious that $F_{ij}$ is homogeneous of degree $-2$ so that $\mcal{Z}$ is homogeneous.

As was noted above, the function $\mcal{Z}$ is unique up to multiplication by functions independent of $x_{i}$ and $x_{j}$.
Let $C_{ij}(\bm{x})=C_{ij}(x_{1},\dots, \hat{x}_{i},\dots, \hat{x}_{j},\dots, x_{N})$ be a function independent of $x_{i}$ and $x_{j}$.
Assuming that $\mcal{Z}$ is annihilated by operators $\mcal{Q}_{i}+F_{ij}(x_{i},\bm{x})$ and $\mcal{Q}_{j}+F_{ij}(x_{j},\bm{x})$,
we also require $C_{ij}(\bm{x})\mcal{Z}(\bm{x})$ to be annihilated by them.
Then, we have
\begin{equation*}
	\sum_{k;k\neq i,j}\frac{1}{x_{i}-x_{k}}\left(\del_{x_{k}}C_{ij}\right)(\bm{x})=0.
\end{equation*}
Multiplying $(x_{i}-x_{k})$ for some $k\neq i,j$ and take the limit $x_{i}\to x_{k}$, we see that $\del_{x_{k}}C_{ij}=0$.
Since $k\neq i,j$ is arbitrary, this means that $C_{ij}$ is a constant.
\end{proof}

Theorem \ref{thm:commutation} can be immediately extended to a family of mutually commutative backward Loewner chains.
Recall that, for each $i=1,\dots, N$, $\left(f^{(i)}_{t}(\cdot):t\ge 0\right)$ is an $i$-th backward SLE$(\kappa_{i},b_{i})$.

\begin{cor}
\label{cor:mutual_commutation}
The backward Loewner chains $\left(f^{(i)}_{t}(\cdot):t\ge 0 \right)$, $i=1,\dots, N$ are mutually commutative if and only if the following conditions are satisfied:
\begin{enumerate}
\item 	There exists $\kappa>0$ such that either $\kappa_{i}=\kappa$ or $\kappa_{i}=16/\kappa$ holds for $i=1,\dots, N$.
\item 	There exists a translation invariant and homogeneous function $\mcal{Z}=\mcal{Z}(x_{1},\dots, x_{N})$ on $\mrm{Conf}_{N}(\mbb{R})$
		with the following properties:
		\begin{enumerate}
		\item 	Each function $b_{i}$ is given by $b_{i}=\kappa_{i}\del_{x_{i}}\log\mcal{Z}$, $i=1,\dots, N$.
		\item 	It satisfies $\mcal{D}^{\bm{\kappa}}_{i}\mcal{Z}=0$, $i=1,\dots, N$, where
				\begin{equation*}
					\mcal{D}^{\bm{\kappa}}_{i}=\frac{\kappa_{i}}{2}\del_{x_{i}}^{2}-2\sum_{j;j\neq i}\left(\frac{1}{x_{j}-x_{i}}\del_{x_{j}}-\frac{h_{\kappa_{j}}}{(x_{j}-x_{i})^{2}}\right),\ \ i=1,\dots, N.
				\end{equation*}
		\end{enumerate}
\end{enumerate}
\end{cor}
\begin{proof}
First, it is immediate from Theorem \ref{thm:commutation} that there exists $\kappa>0$ and
either $\kappa_{i}=\kappa$ or $\kappa_{i}=16/\kappa$ holds for every $i=1,\dots, N$.
In the same way as in the proof of Theorem \ref{thm:commutation}, we have
\begin{equation*}
	\kappa_{i}\del_{x_{i}}b_{j}-\kappa_{j}\del_{x_{j}}b_{i}=0, \quad 1\le i<j\le N.
\end{equation*}
This is equivalent to that the one-form $\omega=\sum_{i=1}^{N}\kappa_{i}^{-1}b_{i}dx_{i}$ is closed.
Hence, there exists a function $\mcal{Z}$ such that $d\log \mcal{Z}=\omega$, in other words, $b_{i}=\kappa_{i}\del_{x_{i}}\log \mcal{Z}$, $i=1,\dots, N$.
Furthermore, the function $\mcal{Z}$ satisfies the system of differential equations
\begin{equation*}
	\left(\frac{\kappa_{i}}{2}\del_{x_{i}}^{2}-\sum_{k;k\neq i}\frac{2}{x_{k}-x_{i}}\del_{x_{k}}+\frac{2h_{\kappa_{j}}}{(x_{j}-x_{i})^{2}}+F_{ij}(x_{i},\bm{x})\right)\mcal{Z}=0
\end{equation*}
for every pair $\{i,j\}\subset \{1,\dots, N\}$.
Here, $F_{ij}=F_{ij}(x,x_{1},\dots, \hat{x}_{i},\dots, \hat{x}_{j},\dots, x_{N})$, $\{i,j\}\subset \{1,\dots, N\}$ are the functions taken in Theorem \ref{thm:commutation}.
Thinking of these equations for a fixed $i$, we see that the function
\begin{equation*}
	G_{i}(\bm{x})=\frac{2h_{\kappa_{j}}}{(x_{j}-x_{i})^{2}}+F_{ij}(x_{i},\bm{x})
\end{equation*}
is independent of $j$.
For simplicity of description, we consider $G_{1}(\bm{x})$.
Let us assume that $G_{1}(\bm{x})$ has the form of
\begin{equation}
\label{eq:G_{1}_expansion_induction}
	G_{1}(\bm{x})=\sum_{i=2}^{p}\frac{2h_{\kappa_{i}}}{(x_{i}-x_{1})^{2}}+\tilde{G}_{1,p}(x_{1},x_{p+1},\dots, x_{N})
\end{equation}
with some $p=2,\dots, N$, where $\tilde{G}_{1,p}$ is a function that is independent of $x_{2},\dots, x_{p}$.
Notice that this assumption is valid for $p=2$ by taking $\tilde{G}_{1,2}(x_{1},x_{3},\dots,x_{N})=F_{12}(x_{1},\bm{x})$.
We can equate (\ref{eq:G_{1}_expansion_induction}) to
\begin{equation*}
	G_{1}(\bm{x})=\frac{2h_{\kappa_{p+1}}}{(x_{p+1}-x_{1})^{2}}+F_{1p+1}(x_{1},\bm{x})
\end{equation*}
and find that
\begin{equation*}
	F_{1p+1}(x,\bm{x})-\sum_{i=2}^{p}\frac{2h_{\kappa_{i}}}{(x_{i}-x_{1})^{2}}=\tilde{G}_{1,p}(x_{1},x_{p+1},\dots,x_{N})-\frac{2h_{\kappa_{p+1}}}{(x_{p+1}-x_{1})^{2}}.
\end{equation*}
Here, the left hand side is independent of $x_{p+1}$, hence, the right hand side is equal to some function $\tilde{G}_{1,p+1}(x_{1},x_{p+2},\dots, x_{N})$ that is independent of $x_{2},\dots,x_{p+1}$.
Consequently, from (\ref{eq:G_{1}_expansion_induction}), we have
\begin{equation*}
	G_{1}(\bm{x})=\sum_{i=2}^{p+1}\frac{2h_{\kappa_{i}}}{(x_{i}-x_{1})^{2}}+\tilde{G}_{1,p+1}(x_{1},x_{p+2},\dots, x_{N}),
\end{equation*}
and can invoke the induction in $p$.
Applying the same argument to other $i$'s, we conclude that
\begin{equation*}
	G_{i}(\bm{x})=\sum_{j;j\neq i}^{N}\frac{2h_{\kappa_{j}}}{(x_{j}-x_{i})^{2}}+\tilde{G}_{i}(x_{i}), \quad i=1,\dots, N,
\end{equation*}
where, for each $i=1,\dots, N$, $\tilde{G}_{i}$ is a function only of $x_{i}$.
Since they have to be homogeneous of degree $-2$, we may write them as $\tilde{G}_{i}(x_{i})=\frac{c_{i}}{x_{i}^{2}}$ with constants $c_{i}$, $i=1,\dots, N$.
Therefore, the function $\mcal{Z}$ satisfies
\begin{equation*}
	\left(\mcal{D}^{\bm{\kappa}}_{i}+\frac{c_{i}}{x_{i}^{2}}\right)\mcal{Z}=0,\ \ i=1,\dots, N,
\end{equation*}
and also is annihilated by
\begin{equation*}
	\left[\mcal{D}^{\bm{\kappa}}_{i}+\tilde{G}_{i},\mcal{D}^{\bm{\kappa}}_{j}+\tilde{G}_{j}\right]
	=\frac{-4}{x_{i}-x_{j}}\left(\frac{c_{i}}{x_{i}^{3}}+\frac{c_{j}}{x_{j}^{3}}\right).
\end{equation*}
Therefore, the above function itself has to vanish, which implies $c_{i}=0$, $i=1,\dots, N$.
\end{proof}

\section{Proposal of multiple backward SLE}
\label{sect:multiple_backward_SLE}
Let $N\in\mbb{N}$ and $\kappa>0$ be fixed.
We think of a multiple backward SLE as a special case of a family of mutually commuting Loewner chains considered in Corollary \ref{cor:mutual_commutation},
where $\kappa_{i}=\kappa$, $i=1,\dots, N$ are chosen all to be equal.
We also write $\mcal{D}^{\kappa}_{i}:=\mcal{D}^{(\kappa,\dots,\kappa)}_{i}$, $i=1,\dots, N$ for simplicity.

Note that the system of differential equations $\mcal{D}^{\kappa}_{i}\mcal{Z}=0$, $i=1,\dots, N$ on a function $\mcal{Z}$ is a one of BPZ equations from two-dimensional CFT \cite{BelavinPolyakovZamolodchikov1984}.
Though this is a system of partial differential equations and the space of its smooth solutions is difficult to study,
these BPZ equations only have regular singular points.
Hence, the general theory for partial differential equations with regular singular points \cite[Appendix B]{Knapp1986} can be applied.
In particular, the space of solutions that admit the following properties is finite dimensional:
\begin{enumerate}
\item[(I)] it is analytic in $\mrm{Conf}_{N}(\mbb{R})$.
\item[(II)] it admits the Frobenius expansion; for any pair $\{i,j\}\subset\{1,\dots, N\}$, there exists $\Delta_{ij}\in\mbb{R}$ such that the solution is expanded as
		\begin{equation*}
			\mcal{Z}(\dots, x_{i},\dots, \overset{j}{\check{x_{i}+\varepsilon}},\dots)=\sum_{n=0}^{N}\varepsilon^{\Delta_{ij}+n}\mcal{Z}_{n}(\dots, x_{i},\dots, \overset{j}{\check{x_{i}}},\dots)
		\end{equation*}
		for sufficiently small $\varepsilon>0$.
\end{enumerate}
Furthermore, it is readily seen that with the differential operators $\mcal{D}^{\kappa}_{i}$, $i=1,\dots, N$, each exponent $\Delta_{ij}$, $\{i,j\}\subset\{1,\dots N\}$ is either $-2/\kappa$ or $-(\kappa+6)/\kappa$. Note that these two exponents do not coincide when $\kappa>0$.
Therefore, according to the general theory, a set of exponents $\{\Delta_{ij}\}_{\{i,j\}\subset\{1,\dots, N\}}$ uniquely determines a solution up to multiplicative constants.

\begin{defn}
\label{defn:N_kappa_partition_function}
An $(N,\kappa)$-partition function $\mcal{Z}=\mcal{Z}(x_{1},\dots, x_{N})$
is a translation invariant and homogenous function on $\mrm{Conf}_{N}(\mbb{R})$ that satisfies the system of differential equations
\begin{equation*}
	\mcal{D}^{\kappa}_{i}\mcal{Z}=0,\ \ i=1,\dots, N,
\end{equation*}
and the properties (I) and (II) described above.
\end{defn}

Given an $(N,\kappa)$-partition function $\mcal{Z}=\mcal{Z}(x_{1},\dots, x_{N})$, what follows is a temporary definition of a multiple backward SLE:
\begin{defn}
\label{defn:multiple_backward_SLE_working}
Let $\bm{X}\in\mrm{Conf}_{N}(\mbb{R})$ be an $N$-point configuration and let $\mcal{Z}$ be an $(N,\kappa)$-partition function.
A $\mcal{Z}$-multiple backward SLE$(\kappa)$ starting at $\bm{X}$ is an $N$-tuple of Loewner chains $\left\{\left(f^{(i)}_{t}(\cdot):t\ge 0\right)\right\}_{i=1}^{N}$,
where each $\left(f_{t}^{(i)}(\cdot):t\ge 0\right)$ is an $i$-th backward SLE$(\kappa,b_{i})$ starting at $\bm{X}$ with $b_{i}=\kappa\del_{x_{i}}\log\mcal{Z}$, $i=1,\dots, N$.
\end{defn}

Owing to Corollary \ref{cor:mutual_commutation}, the members of a $\mcal{Z}$-multiple SLE$(\kappa)$ consistently generate $N$ random curves in $\mbb{H}$.

We can see that each flow $\left(f^{(i)}_{t}(\cdot):t\ge 0\right)$, $i=1,\dots, N$ is obtained as a Girsanov transform of a backward SLE$(\kappa)$.
Let $\left(f_{t}(\cdot):t\ge 0 \right)$ be a backward SLE$(\kappa)$, which satisfies
\begin{equation*}
	\frac{d}{dt}f_{t}(z)=-\frac{2}{f_{t}(z)-W_{t}},\quad t\ge 0,\quad f_{0}(z)=z\in\mbb{H},
\end{equation*}
where $W_{t}=\sqrt{\kappa}B_{t}$, $t\ge 0$ with $(B_{t}:t\ge 0)$ being a standard Brownian motion with respect to a probability measure $\mbb{P}$.
For $x\in\mbb{R}$, we write the law of a standard Brownian motion starting at $x$ as $\mbb{P}^{x}$.
Let $\mcal{Z}=\mcal{Z}(x_{1},\dots, x_{N})$ be an $(N,\kappa)$-partition function.
For $\bm{X}\in\mrm{Conf}_{N}(\mbb{R})$ and $i\in \{1,\dots, N\}$, we set 
\begin{align}
\label{eq:martingale_M}
	M^{(i)}_{\bm{X},t}&:=\prod_{j;j\neq i}\pr{f}_{t}(X_{j})^{h_{\kappa}} \mcal{Z}\left(X_{t}^{(1)},\dots, \overset{i}{\check{W}_{t}},\dots, X^{(N)}_{t}\right),\quad t\ge 0,\\
	\frac{d}{dt}X^{(j)}_{t}&=-\frac{2}{X^{(j)}_{t}-W_{t}},\quad t\ge 0,\quad X^{(j)}_{0}=X_{j},\quad j\neq i \notag
\end{align}
and consider the stochastic process $\left(M^{(i)}_{\bm{X},t}:t\ge 0\right)$ under the probability measure $\mbb{P}^{X_{i}}$.
Here, $\pr{f}_{t}(z)$ is the derivative in terms of $z$. Note that $\pr{f}_{t}(X_{j})>0$, $j\neq i$, hence, there is no ambiguity of defining their non-integer powers.

For each $i\in\{1,\dots, N\}$, $n\in\mbb{N}$ and $\bm{X}\in\mrm{Conf}_{N}(\mbb{R})$, we define a stopping time
\begin{equation*}
	\tau^{(i)}_{\bm{X},n}:=\mrm{inf}\left\{t>0\Big|\left|M^{(i)}_{\bm{X},t}\right|>n\right\}.
\end{equation*}

\begin{thm}
\label{thm:Girsanov_transform}
The stochastic process $\left(M^{(i)}_{\bm{X},t}:t\ge 0\right)$ is a local martingale with respect to $\mbb{P}^{X_{i}}$.
For $n\in\mbb{N}$, define a probability measure $\mbb{Q}_{\bm{X},n}^{(i)}$ by
\begin{equation}
\label{eq:transformed_measure}
	\frac{d\mbb{Q}_{\bm{X},n}^{(i)}}{d\mbb{P}^{X_{i}}}:=\lim_{t\to \infty}\frac{M^{(i)}_{\bm{X},t\wedge \tau^{(i)}_{\bm{X},n}}}{M^{(i)}_{\bm{X},0}}.
\end{equation}
Then, the Loewner chain $\left(f_{t}(\cdot):t\ge 0 \right)$ above is the $i$-th Loewner chain $\left(f^{(i)}_{t}(\cdot):t\ge 0\right)$
of a $\mcal{Z}$-multiple backward SLE$(\kappa)$ starting at $\bm{X}$ under probability measure $\mbb{Q}^{(i)}_{\bm{X},n}$
up to the stopping time $\tau^{(i)}_{\bm{X},n}$.
\end{thm}

\begin{proof}
By It{\^o}'s formula, we see that
\begin{align*}
	dM^{(i)}_{\bm{X},t}
	=&\prod_{j;j\neq i}\pr{f}_{t}(X_{j})^{h_{\kappa}}(\mcal{D}^{\kappa}_{i}\mcal{Z})\left(X_{t}^{(1)},\dots, \overset{i}{\check{W}_{t}},\dots, X^{(N)}_{t}\right)dt \\
	&+\prod_{j;j\neq i}\pr{f}_{t}(X_{j})^{h_{\kappa}}(\del_{x_{i}}\mcal{Z})\left(X_{t}^{(1)},\dots, \overset{i}{\check{W}_{t}},\dots, X^{(N)}_{t}\right)dW_{t},\quad t\ge 0.
\end{align*}
The assumption that $\mcal{Z}$ is an $(N,\kappa)$-partition function ensures that the stochastic process $\left(M^{(i)}_{\bm{X},t}:t\ge 0\right)$
is a local martingale.
Its increment is also written as
\begin{equation}
\label{eq:Girsanov_martingale_SDE}
	dM^{(i)}_{\bm{X},t}=s^{(i)}_{\bm{X},t}M^{(i)}_{\bm{X},t}dB_{t},\quad t\ge 0,
\end{equation}
where
\begin{equation}
\label{eq:Girsanov_drift_term}
	s^{(i)}_{\bm{X},t}:=\sqrt{\kappa}(\del_{x_{i}}\log \mcal{Z})\left(X_{t}^{(1)},\dots, \overset{i}{\check{W}_{t}},\dots, X^{(N)}_{t}\right),\quad t\ge 0.
\end{equation}
Therefore, by Girsanov--Maruyama's theorem, the stochastic process $\left(B^{(i)}_{n,t}:t\ge 0\right)$ defined by
\begin{equation}
\label{eq:transformed_BM}
	B^{(i)}_{n,t}:=B_{t}-\int_{0}^{t\wedge \tau^{(i)}_{\bm{X},n}}s^{(i)}_{\bm{X},u}du,\quad t\ge 0
\end{equation}
is a Brownian motion starting at $X_{i}$ with respect to $\mbb{Q}_{\bm{X},n}^{(i)}$.
It follows that the backward Loewner chain driven by $(W_{t}:t\ge 0)$ under $\mbb{Q}_{\bm{X},n}^{(i)}$ is an $i$-th backward SLE$(\kappa,b)$
with $b=\kappa\del_{x_{i}}\log\mcal{Z}$ up to the stopping time $\tau^{(i)}_{\bm{X},n}$,
which is the $i$-th flow of an $\mcal{Z}$-multiple backward SLE$(\kappa)$.
\end{proof}

Owing to Theorem \ref{thm:Girsanov_transform}, a $\mcal{Z}$-multiple backward SLE$(\kappa)$ is equivalently defined as follows:
\begin{defn}
\label{defn:multiple_backward_SLE}
Let $\kappa>0$, $N\in\mbb{N}$, $\mcal{Z}$ be an $(N,\kappa)$-partition function and $\bm{X}\in\mrm{Conf}_{N}(\mbb{R})$.
A $\mcal{Z}$-multiple backward SLE$(\kappa)$ starting at $\bm{X}$ is 
a family of probability measures $\left\{\mbb{Q}^{(i)}_{\bm{X},n}:i=1,\dots, N, n\in\mbb{N}\right\}$
each of which is defined by (\ref{eq:transformed_measure}).
\end{defn}

We will work with Definition \ref{defn:multiple_backward_SLE} as a definition of a multiple backward SLE.
When we consider a backward Loewner chain $(f_{t}(\cdot):t\ge 0)$ driven by $(W_{t}=\sqrt{\kappa}B_{t}:t\ge 0)$
with $(B_{t}:t\ge 0)$ governed by $\mbb{Q}^{(i)}_{\bm{X},n}$,
it is just the $i$-th Loewner chain $\left(f^{(i)}_{t}(\cdot):t\ge 0\right)$ of a multiple SLE$(\kappa)$
defined in Definition \ref{defn:multiple_backward_SLE_working} up to the stopping time $\tau^{(i)}_{\bm{X},n}$.

\section{Coupling with GFF}
\label{sect:coupling_GFF}
\subsection{Prelimiaries}
Let us make some preliminaries on free boundary GFF.
Expositions of this subject can be found in \cite{Sheffield2016,Berestycki2016, QianWerner2018}.
Let $D\subsetneq\mbb{C}$ be a simply connected domain and $C^{\infty}_{\nabla}(D)$ be the space of real-valued smooth functions on $D$
with square integrable gradients.
We equip it with the Dirichlet inner product $(\cdot,\cdot)_{\nabla}$ defined by
\begin{equation*}
	(f,g)_{\nabla}:=\frac{1}{2\pi}\int_{D}\nabla f \cdot \nabla g,\quad f,g\in C^{\infty}_{\nabla}(D),
\end{equation*}
and denote the induced norm by $\|\cdot\|_{\nabla}=\sqrt{(\cdot,\cdot)_{\nabla}}$.
Since the subspace $\mcal{N}\subset C^{\infty}_{\nabla}(D)$ of constant functions coincides with the radical of this norm,
the quotient space $C^{\infty}_{\nabla}(D)/\mcal{N}$ is a pre-Hilbert space.
We write $[f]:=f+\mcal{N}$, $f\in C^{\infty}_{\nabla}(D)$.
The Hilbert space completion of $C^{\infty}_{\nabla}(D)/\mcal{N}$ by $(\cdot,\cdot)_{\nabla}$ will be denoted by $W(D)$.

\begin{defn}
A free boundary GFF on $D$ is a collection $\{(H,[f])_{\nabla}|[f]\in W(D)\}$ of centered Gaussian random variables labeled by $W(D)$ such that
\begin{equation*}
	\mbb{E}[(H,[f])_{\nabla}(H,[g])_{\nabla}]=(f,g)_{\nabla},\quad [f], [g]\in W(D).
\end{equation*}
We write $\mcal{P}$ for the probability law for these Gaussian random variables.
\end{defn}
This family of Gaussian random variables is constructed by means of Bochner--Minlos's theorem (see e.g. \cite[Chapter 3]{Hida1980}).
Note that the Dirichlet inner product in the right-hand side is independent of the choice of a representative.

Let $\Delta$ be the Neumann boundary Laplacian on $D$ and $\msf{D}((-\Delta)^{-1})$ be the defining domain of $(-\Delta)^{-1}$ in $W(D)$.
Then, we define $(H,[f])^{\prime}:=2\pi (H,(-\Delta)^{-1}[f])_{\nabla}$, $[f]\in \msf{D}((-\Delta)^{-1})$.
The action of $(-\Delta)^{-1}$ is described by means of Green's function.
For $[f]\in \msf{D}((-\Delta)^{-1})$, we can find a unique representative $f\in [f]$ such that $\int_{D} f=0$.
Then, we have
\begin{equation*}
	(-\Delta)^{-1}[f]=\left[\frac{1}{2\pi}\int_{D}G(z,w)f(w)dw\right],
\end{equation*}
where $G(z,w)$, $z,w\in D$ is Neumann boundary Green's function on $D$.
Motivated by this, we set
\begin{equation*}
	\mcal{C}_{0}(D):=\left\{f\in C^{\infty}_{\nabla}(D)\Bigg|\int_{D}f=0\right\}
\end{equation*}
and
\begin{equation*}
	(H,f):=(H,[f])^{\prime},\quad f\in\mcal{C}_{0}(D).
\end{equation*}
Then, the collection $\{(H,f)|f\in \mcal{C}_{0}(D)\}$ is a one of centered Gaussian random variables such that
\begin{equation*}
	\mbb{E}[(H,f)(H,g)]=\int_{D\times D}f(z)G(z,w)g(w)dzdw,\quad f,g\in \mcal{C}_{0}(D).
\end{equation*}

It is natural to think of $H$ as a random distribution with test functions taken from $\mcal{C}_{0}(D)$
to symbolically write
\begin{equation*}
	(H,f)=\int_{D}H(z)f(z)dz,\quad f\in \mcal{C}_{0}(D).
\end{equation*}
We understand the object $H(z)$, $z\in D$ in this sense and also call $H$ a free boundary GFF on $D$.
The covariance structure is reproduced by the formula
\begin{equation*}
	\mbb{E}[H(z)H(w)]=G(z,w),\quad z,w\in D,\quad z\neq w.
\end{equation*}

\begin{exam}
In the case that $D=\mbb{H}$ is the complex upper half plane, we set
\begin{equation*}
	G_{\mbb{H}}(z,w):=-\log |z-w| -\log |z-\overline{w}|,\ \ z,w\in\mbb{H},\ \ z\neq w
\end{equation*}
as Neumann boundary Green's function on $\mbb{H}$.
\end{exam}

Free boundary GFF plays the role of an ingredient of the Liouville quantum gravity \cite{Polyakov1981a, Polyakov1981b} and a probability theoretical construction of 
Liouville CFT.
This aspect of GFF has been studied extensively \cite{DuplantierSheffield2009,DuplantierSheffield2011,DuplantierMillerSheffield2014,RhodesVargas2017,DavidKupiainenRhodesVargas2016,GuillarmouRhodesVargas2016,DavidRhodesVargas2016,GwynneMillerSheffield2017,HuangRhodesVargas2018,KupiainenRhodesVargas2019}.

\subsection{SLE/GFF-coupling}
Let us begin with a definition of boundary perturbation for free boundary GFF.
Here we fix $N\in\mbb{N}$.
\begin{defn}
Let $u(\cdot;x_{1},\dots, x_{N})=u(z;x_{1},\dots, x_{N})$ be a harmonic function of $z\in\mbb{H}$
with additional parameters $(x_{1},\dots, x_{N})\in\mrm{Conf}_{N}(\mbb{R})$.
We say that $u(\cdot;x_{1},\dots, x_{N})$ is a boundary perturbation for free boundary GFF if the following conditions are satisfied:
\begin{description}
\item[Translation invariance]
		For any $a\in\mbb{R}$,
		\begin{equation*}
			u(z+a;x_{1}+a,\dots, x_{N}+a)\equiv u(z;x_{1},\dots, x_{N}),\quad z\in\mbb{H},\quad (x_{1},\dots, x_{N})\in\mrm{Conf}_{N}(\mbb{R})
		\end{equation*}
		modulo additive constants.
\item[Scale invariance]
		For any $\lambda>0$,
		\begin{equation*}
			u(\lambda z; \lambda x_{1},\dots, \lambda x_{N})\equiv u(z;x_{1},\dots, x_{N}),\quad z\in\mbb{H},\quad (x_{1},\dots, x_{N})\in \mrm{Conf}_{N}(\mbb{R})
		\end{equation*}
		modulo additive constants.
\end{description}
\end{defn}

For a boundary perturbation $u(\cdot;x_{1},\dots, x_{N})$,
one can think of a random distribution $H_{(u,\bm{X})}:=H+u(\cdot;X_{1},\dots, X_{N})$ on $\mbb{H}$,
where $H$ is a free boundary GFF on $\mbb{H}$ and $\bm{X}=(X_{1},\dots, X_{N})\in\mrm{Conf}_{N}(\mbb{R})$.
We call the above $H_{(u,\bm{X})}$ a $(u,\bm{X})$-perturbed free boundary GFF.
Note that a free boundary GFF $H$ and a $(u,\bm{X})$-perturbed free boundary GFF $H_{(u,\bm{X})}$
cannot be distinguished by test functions in the {\it bulk}.
In fact, since $u=u(z;X_{1},\dots, X_{N})$ is harmonic in $z\in\mbb{H}$,
for a test function $f\in\mcal{C}_{0}(\mbb{H})$ that is supported in $\mbb{H}$, we have
$(H,f)=(H_{(u,\bm{X})},f)$ a.s.

Suppose that an $(N,\kappa)$-partition function $\mcal{Z}=\mcal{Z}(x_{1},\dots, x_{N})$ is given.
Let $\mbb{P}$ be the law of a backward SLE$(\kappa)$ $\left(f_{t}(\cdot):t\ge 0\right)$ that is independent of a GFF
and let $\left\{\mbb{Q}^{(i)}_{\bm{X}}:i=1,\dots, N\right\}$ be the family of laws of a $\mcal{Z}$-multiple backward SLE$(\kappa)$
starting at $\bm{X}\in\mrm{Conf}_{N}(\mbb{R})$ defined in (\ref{eq:transformed_measure}).
For each $i\in\{1,\dots, N\}$, we consider the stochastic distribution $\left(\mfrak{h}_{\bm{X},t}^{(i)}:t\ge 0\right)$ defined by
\begin{align*}
	\mfrak{h}_{\bm{X},t}^{(i)}(z)&:=u\left(f_{t}(z);X^{(1)}_{t},\dots, \overset{i}{\check{W_{t}}}, \dots, X^{(N)}_{t}\right) + Q\log |\pr{f}_{t}(z)|,\quad t\ge 0,\quad z\in\mbb{H},\\
	\frac{d}{dt}X^{(j)}_{t}&=-\frac{2}{X^{(j)}_{t}-W_{t}},\quad t\ge 0,\quad X^{(j)}_{0}=X_{j},\quad j\neq i,
\end{align*}
where $W_{t}=\sqrt{\kappa}B_{t}$, $t\ge 0$ with $(B_{t}:t\ge 0)$ being a $\mbb{P}^{X_{i}}$-Brownian motion
and we set $Q=\frac{2}{\gamma}+\frac{\gamma}{2}$, $\gamma \in (0,2]$.

\begin{defn}
We say that the $\mcal{Z}$-multiple backward SLE is coupled with a $(u,\bm{X})$-perturbed free boundary GFF $H_{(u,\bm{X})}$ with coupling constant $\gamma$
if, for all $i\in \{1,\dots, N\}$ and $n\in\mbb{N}$,
the stochastic distribution $\left(\mfrak{h}_{\bm{X},t}^{(i)}:t\ge 0\right)$ is a $\mbb{Q}^{(i)}_{\bm{X},n}$-local martingale
with cross variation given by
\begin{equation*}
	d\left[\mfrak{h}^{(i)}_{\bm{X}}(z),\mfrak{h}^{(i)}_{\bm{X}}(w)\right]_{t}=-dG_{t}(z,w),\quad t\ge 0,\quad z,w\in\mbb{H},
\end{equation*}
where $G_{t}(z,w):=G_{\mbb{H}}(f_{t}(z),f_{t}(w))$, $t\ge 0$, $z,w\in\mbb{H}$, $z\neq w$ with
$(f_{t}(\cdot):t\ge 0)$ being a Loewner chain obeying $\mbb{Q}^{(i)}_{\bm{X},n}$.
\end{defn}

This definition is motivated by the following fact.
For each $i\in\{1,\dots, N\}$, let us set
\begin{equation*}
	\mfrak{p}^{(i)}_{\bm{X},t}:=\mfrak{h}^{(i)}_{\bm{X},t}+H\circ f_{t},\quad t\ge 0.
\end{equation*}
Note that $\mfrak{p}_{\bm{X},0}^{(i)}=H_{(u,\bm{X})}$ regardless of $i\in \{1,\dots, N\}$.

\begin{prop}
Suppose that a $\mcal{Z}$-multiple backward SLE$(\kappa)$ starting at $\bm{X}\in\mrm{Conf}_{N}(\mbb{R})$
is coupled with a $(u,\bm{X})$-perturbed free boundary GFF $H_{(u,\bm{X})}$ with coupling constant $\gamma$
and let $\left(\mfrak{p}_{\bm{X},t}^{(i)}:t\ge 0\right)$, $i\in \{1,\dots, N\}$ be as above.
Then, at each time $t\ge 0$, the law of $\mfrak{p}_{\bm{X},t}^{(i)}$ under $\mcal{P}\otimes \mbb{Q}^{(i)}_{\bm{X},n}$ is identical to that of $H_{(u,\bm{X})}$
under $\mcal{P}$ for every $i=1,\dots, N$ and $n\in\mbb{N}$.
\end{prop}
\begin{proof}
Firstly, we note that a $(u,\bm{X})$-perturbed free boundary GFF gives
Gaussian random variables $(H_{(u,\bm{X})},\rho)$, $\rho\in \mcal{C}_{0}(\mbb{H})$ with mean being shifted by $(u(\cdot;\bm{X}),\rho)$
and variance
\begin{equation*}
	E(\rho)=\int_{\mbb{H}\times\mbb{H}}\rho(z)G_{\mbb{H}}(z,w)\rho(w)dzdw.
\end{equation*}
Therefore, we have
\begin{equation*}
	\mbb{E}\left[e^{\sqrt{-1}\zeta(H_{(u,\bm{X})},\rho)}\right]=e^{\sqrt{-1}\zeta (u(\cdot;\bm{X}),\rho)-\frac{\zeta^{2}}{2}E(\rho)},\quad \rho\in \mcal{C}_{0}(\mbb{H}),\quad \zeta\in\mbb{R}.
\end{equation*}

Let $(\mcal{F}_{t})_{t\ge 0}$ be the filtration associated with a $\mbb{P}^{X_{i}}$-Brownian motion $(B_{t}:t\ge 0)$.
Then, we have
\begin{align*}
	\mbb{E}\left[e^{\sqrt{-1}\zeta(\mfrak{p}^{(i)}_{\bm{X},t},\rho)}\right]
	&=\mbb{E}\left[e^{\sqrt{-1}\zeta (\mfrak{h}^{(i)}_{\bm{X},t},\rho)}\mbb{E}\left[e^{\sqrt{-1}\zeta (H\circ f_{t},\rho)}\Big|\mcal{F}_{t}\right]\right] \\
	&=\mbb{E}\left[e^{\sqrt{-1}\zeta (\mfrak{h}^{(i)}_{\bm{X},t},\rho)-\frac{\zeta^{2}}{2}E_{t}(\rho)}\right],\quad \rho\in\mcal{C}_{0}(\mbb{H}),\quad \zeta\in\mbb{R}
\end{align*}
where we set
\begin{equation*}
	E_{t}(\rho)=\int_{\mbb{H}\times\mbb{H}}\rho(z)G_{t}(z,w)\rho(w)dzdw,\quad \rho\in \mcal{C}_{0}(\mbb{H}).
\end{equation*}
By assumption, we have the quadratic variation of $\left((\mfrak{h}^{(i)}_{\bm{X},t},\rho): t\geq 0\right)$ as $d\left[(\mfrak{h}^{(i)}_{\bm{X}},\rho)\right]_{t}=-dE_{t}(\rho)$, $t\ge 0$, which gives
\begin{equation*}
	\left[(\mfrak{h}^{(i)}_{\bm{X}},\rho)\right]_{t}=-E_{t}(\rho)+E(\rho),\quad t\ge 0.
\end{equation*}
This leads to
\begin{align*}
	\mbb{E}\left[e^{\sqrt{-1}\zeta(\mfrak{p}^{(i)}_{\bm{X},t},\rho)}\right]
	&=e^{-\frac{\zeta^{2}}{2}E(\rho)}\mbb{E}\left[e^{\sqrt{-1}\zeta (\mfrak{h}^{(i)}_{\bm{X},t},\rho)+\frac{\zeta^{2}}{2}[(\mfrak{h}^{(i)}_{\bm{X}},\rho)]_{t}}\right],
\end{align*}
where $\left(e^{\sqrt{-1}\zeta (\mfrak{h}^{(i)}_{\bm{X},t},\rho)+\frac{\zeta^{2}}{2}[(\mfrak{h}^{(i)}_{\bm{X}},\rho)]_{t}}:t\ge 0\right)$ is a martingale.
Therefore, we have
\begin{equation*}
	\mbb{E}\left[e^{\sqrt{-1}\zeta(\mfrak{p}^{(i)}_{\bm{X},t},\rho)}\right]=e^{\sqrt{-1}\zeta (u(\cdot;\bm{X}),\rho)-\frac{\zeta^{2}}{2}E(\rho)}
	=\mbb{E}\left[e^{\sqrt{-1}\zeta (H_{(u,\bm{X})},\rho)}\right],\quad  \rho\in \mcal{C}_{0}(\mbb{H}), \quad \zeta\in\mbb{R},
\end{equation*}
which gives the desired result.
\end{proof}
This proposition is interpreted in terms of conformal welding of quantum surfaces \cite{Sheffield2016,KatoriKoshida2020a}.
Indeed, the Loewner chain $\left(f_{t}(\cdot):t\ge 0 \right)$ under the law $\mbb{Q}^{(i)}_{\bm{X},n}$ gives the welding map around the $i$-th point $X_{i}$
up to the stopping time $\tau^{(i)}_{\bm{X},n}$.

The main theorem goes as follows.
\begin{thm}
\label{thm:coupling_constraint}
Let $N\in\mbb{N}$, $\kappa>0$, $\mcal{Z}$ be an $(N,\kappa)$-partition function and $u=u(z;x_{1},\dots, x_{N})$ be a boundary perturbation for free boundary GFF.
A $\mcal{Z}$-multiple backward SLE$(\kappa)$ starting at $\bm{X}\in\mrm{Conf}_{N}(\mbb{R})$
is coupled with a $(u,\bm{X})$-perturbed free boundary GFF $H_{(u,\bm{X})}$ with coupling constant $\gamma$
for an arbitrary initial condition $\bm{X}\in\mrm{Conf}_{N}(\mbb{R})$
if and only if the following conditions are satisfied:
\begin{enumerate}
\item 	The relation between parameters $\sqrt{\kappa}=\gamma$ or $\sqrt{\kappa}=4/\gamma$ holds.
\item 	The $(N,\kappa)$-partition function is given by
		\begin{equation*}
			\mcal{Z}(x_{1},\dots, x_{N})=\prod_{1\le i<j\le N}|x_{i}-x_{j}|^{-2/\kappa}
		\end{equation*}
		up to multiplicative constants.
\item 	The boundary perturbation $u(\cdot;x_{1},\dots, x_{N})$ is given by
		\begin{equation*}
			u(z;x_{1},\dots,x_{N})=\frac{2}{\sqrt{\kappa}}\sum_{i=1}^{N}\log |z-x_{i}|,\quad z\in\mbb{H}
		\end{equation*}
		up to additive constants.
\end{enumerate}
\end{thm}

Before proving Theorem \ref{thm:coupling_constraint}, let us note the following fact.
\begin{lem}
\label{lem:coupling_condition}
A $\mcal{Z}$-multiple backward SLE$(\kappa)$ starting at $\bm{X}\in\mrm{Conf}_{N}(\mbb{R})$ is coupled with $H_{(u,\bm{X})}$ with coupling constant $\gamma$
if and only if there exists a sequence $\bm{\epsilon}=\left(\epsilon_{i}\in\{\pm1 \}:i=1,\dots N\right)$ such that,
for each $i\in\{1,\dots, N\}$, the increment of $\left(\mfrak{h}^{(i)}_{\bm{X},t}:t\ge 0\right)$ is given by
\begin{equation}
\label{eq:increment_h_equivalent_to_coupling}
	d\mfrak{h}^{(i)}_{\bm{X},t}(z)=\mrm{Re}\frac{2\epsilon_{i}}{f_{t}(z)-W_{t}}dB^{(i)}_{n,t},\quad t\ge 0,\quad z\in\mbb{H},
\end{equation}
for every $n\in\mbb{N}$,
where $\left(B^{(i)}_{n,t}:t\ge 0\right)$ is a $\mbb{Q}^{(i)}_{\bm{X},n}$-Brownian motion defined by (\ref{eq:transformed_BM}).
\end{lem}
\begin{proof}
It follows from a direct computation of the increment of the stochastic process $(G_{t}(z,w):t\ge 0)$, $z,w\in\mbb{H}$, $z\neq w$ that
\begin{equation*}
	dG_{t}(z,w)=-\mrm{Re}\frac{2}{f_{t}(z)-W_{t}}\mrm{Re}\frac{2}{f_{t}(w)-W_{t}} dt,\ \ t\ge 0,\ \ z,w\in\mbb{H}.
\end{equation*}
Therefore, it is obvious that (\ref{eq:increment_h_equivalent_to_coupling}) implies the coupling.
Conversely, let us assume the coupling. Then, for each $i=1,\dots, N$, the increment of the stochastic process $\left(\mfrak{h}^{(i)}_{\bm{X},t}:t\ge 0\right)$ has the form of
\begin{equation*}
	d\mfrak{h}^{(i)}_{\bm{X},t}(z)=F^{(i)}_{n,t}(z)dB^{(i)}_{n,t},\quad t\ge 0,\quad z\in\mbb{H}
\end{equation*}
with some stochastic process $(F^{(i)}_{n,t}(z):t\geq 0)$ depending on $z\in\mbb{H}$ for every $n\in\mbb{N}$.
From the assumption on the cross variations, we have
\begin{equation}
\label{eq:relations_coefficients_processes}
	F^{(i)}_{n,t}(z)F^{(i)}_{n,t}(w)=\mrm{Re}\frac{2}{f_{t}(z)-W_{t}}\mrm{Re}\frac{2}{f_{t}(w)-W_{t}}, \quad z,w\in \mbb{H},
\end{equation}
which implies that there exists a stochastic process $(U^{(i)}_{n,t}:t\geq 0)$ independent of $z\in \mbb{H}$ such that
\begin{equation*}
	U^{(i)}_{n,t}=\frac{F^{(i)}_{n,t}(z)}{\mrm{Re}\frac{2}{f_{t}(z)-W_{t}}}=\frac{F^{(i)}_{n,t}(w)}{\mrm{Re}\frac{2}{f_{t}(w)-W_{t}}}, \quad t\geq 0.
\end{equation*}
Back to (\ref{eq:relations_coefficients_processes}), we must have $(U^{(i)}_{n,t})^{2}=1$, $t\geq 0$, which gives the desired result.
\end{proof}

\begin{proof}[Proof of Theorem \ref{thm:coupling_constraint}]
For a boundary perturbation $u(\cdot;x_{1},\dots, x_{N})=u(z;x_{1},\dots, x_{N})$, we write
its holomorphic extension by $\tilde{u}(z;x_{1},\dots, x_{N})$, in other words, we have
\begin{equation*}
	u(z;x_{1},\dots, x_{N})=\mrm{Re}\tilde{u}(z;x_{1},\dots,x_{N}),\quad z\in \mbb{H}.
\end{equation*}
Such a holomorphic function uniquely exists on $\mbb{H}$ up to additive constants.
Then, the stochastic process $(\mfrak{h}^{(i)}_{\bm{X},t}:t\ge 0)$ is also realized as
$\mfrak{h}^{(i)}_{\bm{X}, t}=\mrm{Re}\tilde{\mfrak{h}}^{(i)}_{\bm{X}, t}$, $t\ge 0$ where
\begin{equation*}
	\tilde{\mfrak{h}}_{\bm{X},t}^{(i)}(z):=\tilde{u}\left(f_{t}(z);X^{(1)}_{t},\dots, \overset{i}{\check{W_{t}}}, \dots, X^{(N)}_{t}\right) + Q\log \pr{f}_{t}(z),\quad t\ge 0,\quad z\in\mbb{H}.
\end{equation*}
By definition of the probability measure $\mbb{Q}^{(i)}_{\bm{X},n}$ in (\ref{eq:transformed_measure}), the stochastic process $\left(\tilde{\mfrak{h}}^{(i)}_{\bm{X},t}:t\ge 0\right)$
is a $\mbb{Q}^{(i)}_{\bm{X},n}$-local martingale if and only if the stochastic process 
$\left(N^{(i)}_{\bm{X},t}:=\tilde{\mfrak{h}}^{(i)}_{\bm{X},t}M^{(i)}_{\bm{X},t}:t\ge 0\right)$ is a $\mbb{P}^{X_{i}}$-local martingale.
For convenience, we set
\begin{align}
\label{eq:def_function_X}
	\mcal{X}(z;x_{1},\dots, x_{N}) :=&\tilde{u}(z;x_{1},\dots, x_{N})\mcal{Z}(x_{1},\dots, x_{N}),\\
	& z\in\mbb{H},\quad (x_{1},\dots, x_{N})\in\mrm{Conf}_{N}(\mbb{R}). \notag
\end{align}
Then, the stochastic process $\left(N^{(i)}_{\bm{X},t}:t\ge 0\right)$ is explicitly written as
\begin{align}
\label{eq:martingale_N}
	N^{(i)}_{\bm{X},t}(z)=&\prod_{j;j\neq i}\pr{f}_{t}(X_{j})^{h_{\kappa}}\mcal{X}\left(f_{t}(z);X^{(1)}_{t},\dots, \overset{i}{\check{W_{t}}},\dots, X^{(N)}_{t}\right) \\
	&+Q\log \pr{f}_{t}(z) \prod_{j;j\neq i}\pr{f}_{t}(X_{j})^{h_{\kappa}}\mcal{Z}\left(X^{(1)}_{t},\dots, \overset{i}{\check{W}_{t}},\dots, X^{(N)}_{t}\right),\quad t\ge 0,\quad z\in\mbb{H}. \notag
\end{align}
Its increment is computed as
\begin{align*}
	dN^{(i)}_{\bm{X},t}(z)
	=&\prod_{j;j\neq i}\pr{f}_{t}(X_{j})^{h_{\kappa}}\Biggl[(\mcal{D}_{z,i}^{\kappa}\mcal{X})\left(f_{t}(z);X^{(1)}_{t},\dots, \overset{i}{\check{W}_{t}},\dots, X^{(N)}_{t}\right)\\
	&\hspace{70pt}+\frac{2Q}{(f_{t}(z)-W_{t})^{2}}\mcal{Z}\left(X^{(1)}_{t},\dots,\overset{i}{\check{W}_{t}},\dots, X^{(N)}_{t}\right)\Biggr]dt \\
	&+\Biggl[\sqrt{\kappa}\prod_{j;j\neq i}\pr{f}_{t}(X_{j})^{h_{\kappa}}(\del_{x_{j}}\mcal{X})\left(f_{t}(z),X^{(1)}_{t},\dots, \overset{i}{\check{W}_{t}},\dots, X^{(N)}_{t}\right)\\
	&\hspace{20pt}+Q\log \pr{f}_{t}(z)s^{(i)}_{\bm{X},t}M^{(i)}_{\bm{X},t}\Biggr] dB_{t},\quad t\ge 0,\quad z\in\mbb{H},
\end{align*}
where
\begin{equation*}
	\mcal{D}_{z,i}^{\kappa}:=\frac{\kappa}{2}\del_{x_{i}}^{2}-2\sum_{j;j\neq i}\left(\frac{1}{x_{j}-x_{i}}\del_{x_{j}}-\frac{h_{\kappa}}{(x_{j}-x_{i})^{2}}\right)-\frac{2}{z-x_{i}}\del_{z},\quad i=1,\dots, N.
\end{equation*}
Recall that the stochastic process denoted by $\left(s^{(i)}_{\bm{X},t}:t\geq 0\right)$ was defined in (\ref{eq:Girsanov_drift_term}).
Therefore, the stochastic process $\left(N^{(i)}_{\bm{X},t}:t\ge 0\right)$ is a $\mbb{P}^{X_{i}}$-local martingale
for an arbitrary initial condition $\bm{X}\in\mrm{Conf}_{N}(\mbb{R})$ if and only if the function $\mcal{X}$ satisfies
\begin{align}
\label{eq:defining_relation_X}
	(\mcal{D}_{z,i}^{\kappa}\mcal{X})(z;x_{1},\dots, x_{N})+\frac{2Q}{(z-x_{i})^{2}}\mcal{Z}(x_{1},\dots, x_{N})=0,& \\
	z\in\mbb{H},\quad (x_{1},\dots, x_{N})\in\mrm{Conf}_{N}(\mbb{R})&. \notag
\end{align}

Assuming (\ref{eq:defining_relation_X}), we write the increment of $\left(N^{(i)}_{\bm{X},t}:t\ge 0\right)$ as
\begin{equation}
\label{eq:increment_N}
	dN^{(i)}_{\bm{X},t}(z)=\alpha^{(i)}_{\bm{X},t}(z)dB_{t},\quad t\ge 0,\quad z\in\mbb{H}
\end{equation}
with
\begin{align*}
	\alpha^{(i)}_{\bm{X},t}(z):=&\sqrt{\kappa}\prod_{j;j\neq i}\pr{f}_{t}(X_{j})^{h_{\kappa}}(\del_{x_{j}}\mcal{X})\left(f_{t}(z),X^{(1)}_{t},\dots, \overset{i}{\check{W}_{t}},\dots, X^{(N)}_{t}\right) \\
	&+Q\log \pr{f}_{t}(z)s^{(i)}_{\bm{X},t}M^{(i)}_{\bm{X},t},\quad t\ge 0,\quad z\in\mbb{H}.
\end{align*}

Let us get back to consideration of the stochastic process $\left(\tilde{\mfrak{h}}^{(i)}_{\bm{X},t}:t\ge 0\right)$.
Recall the relation $\tilde{\mfrak{h}}^{(i)}_{\bm{X},t}(z)=N^{(i)}_{\bm{X},t}(z)/M^{(i)}_{\bm{X},t}$, $t\geq 0$, $z\in \mbb{H}$.
The increment of the numerator is given in (\ref{eq:increment_N}). It also follows from (\ref{eq:Girsanov_martingale_SDE}) that
\begin{equation*}
	d\left(\frac{1}{M^{(i)}_{\bm{X},t}}\right)=-\frac{s^{(i)}_{\bm{X},t}}{M^{(i)}_{\bm{X},t}}\left(dB_{t}-s^{(i)}_{\bm{X},t}\right), \quad t\geq 0.
\end{equation*}
Then, the increment of the stochastic process $\left(\tilde{\mfrak{h}}^{(i)}_{\bm{X},t}:t\ge 0\right)$ is computed as
\begin{equation*}
	d\tilde{\mfrak{h}}^{(i)}_{\bm{X},t}(z)=\left(\frac{\alpha^{(i)}_{\bm{X},t}(z)}{M^{(i)}_{\bm{X},t}}-\frac{N^{(i)}_{\bm{X},t}(z)s^{(i)}_{\bm{X},t}}{M^{(i)}_{\bm{X},t}}\right)\left(dB_{t}-s^{(i)}_{\bm{X},t}dt\right),\quad t\ge 0,\quad z\in\mbb{H}.
\end{equation*}
By definition (\ref{eq:transformed_BM}) of the $\mbb{Q}^{(i)}_{\bm{X},n}$-Brownian motion $\left(B^{(i)}_{n,t},t\ge 0\right)$,
we have $dB_{n,t}^{(i)}=dB_{t}-s^{(i)}_{\bm{X},t}dt$, $t\ge 0$.
The coefficient can also be further computed to give
\begin{equation*}
	d\tilde{\mfrak{h}}^{(i)}_{\bm{X},t}(z)=\sqrt{\kappa}\left(\del_{x_{i}}\tilde{u}\right)\left(f_{t}(z);X^{(1)}_{t},\dots,\overset{i}{\check{W}_{t}},\dots, X^{(N)}_{t}\right)dB^{(i)}_{n,t},\quad t\ge 0,\quad z\in\mbb{H}.
\end{equation*}
From Lemma \ref{lem:coupling_condition}, we can also require that there exists a sequence $\bm{\epsilon}=(\epsilon_{i}\in \{\pm\}:i=1,\dots, N)$ such that
\begin{equation}
\label{eq:condition_perturbation}
	(\del_{x_{i}}\tilde{u})(z;x_{1},\dots, x_{N})=\frac{2\epsilon_{i}/\sqrt{\kappa}}{z-x_{i}},\quad z\in \mbb{H},\quad (x_{1},\dots, x_{N})\in\mrm{Conf}_{N}(\mbb{R}),\quad i=1,\dots, N,
\end{equation}
so that the $\mcal{Z}$-multiple backward SLE$(\kappa)$ is coupled with $H_{(u,\bm{X})}$ for an arbitrary $\bm{X}\in\mrm{Conf}_{N}(\mbb{R})$.
The differential equations (\ref{eq:condition_perturbation}) are solved by
\begin{equation*}
	\tilde{u}(z;x_{1},\dots, x_{N})=-\frac{2}{\sqrt{\kappa}}\sum_{i=1}^{N}\epsilon_{i}\log (z-x_{i})+h(z),\quad z\in\mbb{H},\quad (x_{1},\dots, x_{N})\in\mrm{Conf}_{N}(\mbb{R}),
\end{equation*}
where $h=h(z)$ is a holomorphic function only of $z\in\mbb{H}$.
It can be seen that the assumption that $u=\mrm{Re}\tilde{u}$ is a boundary perturbation for free boundary GFF requires the function $h$ to be constant
so that it is translation and scale invariant modulo additive constants.

Let us write $\mcal{X}=\tilde{u}\mcal{Z}$ with $\tilde{u}$ being given above
and apply the operators $\mcal{D}_{z,i}^{\kappa}$, $i=1,\dots, N$ on both sides.
Note that $\mcal{D}_{z,i}^{\kappa}=\mcal{D}^{\kappa}_{i}-\frac{2}{z-x_{i}}\del_{z}$ and $\mcal{D}^{\kappa}_{i}\mcal{Z}=0$, $i=1,\dots, N$.
Then we have, for each $i=1,\dots N$,
\begin{align*}
	(\mcal{D}_{z,i}^{\kappa}\mcal{X})(z;x_{1},\dots, x_{N})
		=&\left(\frac{(\sqrt{\kappa}+4/\sqrt{\kappa})\epsilon_{i}}{(z-x_{i})^{2}}+\frac{4/\sqrt{\kappa}}{z-x_{i}}\sum_{j;j\neq i}\frac{\epsilon_{j}}{x_{i}-x_{j}}\right)\mcal{Z}(x_{1},\dots, x_{N}) \\
		&+\frac{2\sqrt{\kappa}\epsilon_{i}}{z-x_{i}}(\del_{x_{i}}\mcal{Z})(x_{1},\dots, x_{N}),\quad z\in\mbb{H},\quad (x_{1},\dots, x_{N})\in\mrm{Conf}_{N}(\mbb{R}).
\end{align*}
For (\ref{eq:defining_relation_X}) to be satisfied, we must take $\sqrt{\kappa}=\gamma$ or $\sqrt{\kappa}=4/\gamma$ and $\epsilon_{i}=-1$, $i=1,\dots, N$.
We see that additional conditions on the $(N,\kappa)$-partition function are imposed so that
\begin{equation}
\label{eq:additional_condition}
	(\del_{x_{i}}\mcal{Z})(x_{1},\dots, x_{N})=\sum_{j;j\neq i}\frac{-2/\kappa}{x_{i}-x_{j}}\mcal{Z}(x_{1},\dots, x_{N}),\quad (x_{1},\dots, x_{N})\in\mrm{Conf}_{N}(\mbb{R}),\quad i=1,\dots, N.
\end{equation}
This implies that the $(N,\kappa)$-partition function exhibits the asymptotic behavior
\begin{equation*}
	\mcal{Z}(x_{1},\cdots, x_{N})\sim (x_{i}-x_{j})^{-2/\kappa}\quad \mbox{as}\quad x_{i}\downarrow x_{j}
\end{equation*}
for any pair $i, j\in \{1,\dots, N\}$.
It is readily seen that the function
\begin{equation*}
	\mcal{Z}(x_{1},\dots, x_{N})=\prod_{1\le i<j\le N}|x_{i}-x_{j}|^{-2/\kappa},\quad (x_{1},\dots, x_{N})\in\mrm{Conf}_{N}(\mbb{R})
\end{equation*}
is an $(N,\kappa)$-partition function.
Due to the comments before Definition \ref{defn:N_kappa_partition_function}, an $(N,\kappa)$-partition function is uniquely determined up to multiplicative constants
by asymptotic behaviors when any two points approach each other.
Therefore, the $(N,\kappa)$-partition function under consideration is the above one up to multiplicative constants.
\end{proof}

\appendix
\section{Conformal field theory approach}
\label{app:CFT}
In this paper, we avoided an explicit use of CFT.
The idea of the so-called SLE/CFT-correspondence \cite{BauerBernard2003, BauerBernard2004} is to construct several local martingales related to SLE as matrix elements of operator valued distribution in CFT.
We have dealt with several stochastic processes, some of which are local martingales, related to backward SLE in the probability theoretical language.
For readers familiar with CFT, however, it might be more useful to interpret those stochastic processes in the language of CFT.

The free boson field $\phi(z)$ is defined as a formal series:
\begin{equation*}
	\phi (z)=q+a_{0}\log (z)-\sum_{n\neq 0}\frac{a_{n}}{n}z^{-n},
\end{equation*}
where the symbols $q$ and $a_{n}$, $n\in\mbb{Z}$ are subject to the commutation relations:
\begin{equation}
\label{eq:Heisenberg_commutation_relations}
	[a_{m},q]=\delta_{m,0},\quad m\in\mbb{Z},\qquad  [a_{m},a_{n}]=m\delta_{m+n,0},\quad m,n\in\mbb{Z}.
\end{equation}
Here, $\delta_{i,j}$ is the Kronecker delta.
Then, the current field $J(z):=\del \phi(z)=\sum_{n\in\mbb{Z}}a_{n}z^{-n-1}$ satisfies the following operator product expansion (OPE):
\begin{equation*}
	J(z)J(w)\sim \frac{1}{(z-w)^{2}}.
\end{equation*}
The vertex operator $V_{\alpha}(z)$ of charge $\alpha\in\mbb{C}$ is defined by
\begin{align*}
	V_{\alpha}(z)&:=\no{e^{\sqrt{-1}\alpha \phi(z)}}\\
	&=e^{\sqrt{-1}\alpha q}z^{\sqrt{-1}\alpha a_{0}}\exp\left(-\sqrt{-1}\alpha\sum_{j<0}\frac{a_{j}}{j}z^{-j}\right)\exp\left(-\sqrt{-1}\alpha\sum_{j>0}\frac{a_{j}}{j}z^{-j}\right).
\end{align*}
Recall that the normally ordered product $\no{\bullet}$ is defined by
\begin{equation*}
	\no{q^{n}a_{m_{1}}\cdots a_{m_{k}}}:=q^{n}a_{m_{\sigma (1)}}\cdots a_{m_{\sigma (k)}},\quad n, m_{1},\dots, m_{k}\in \mbb{Z},
\end{equation*}
where $\sigma$ is a permutation of $(m_{1},\dots, m_{k})\in \mbb{Z}^{k}$ such that $m_{\sigma (1)}\leq \cdots \leq m_{\sigma (k)}$.
The above definition is independent of the choice of such a permutation because of the commutation relations (\ref{eq:Heisenberg_commutation_relations}).
Note that the free boson field is also obtained formally as
\begin{equation*}
	\phi(z)=-\sqrt{-1}\frac{d}{d\alpha}\Big|_{\alpha=0}V_{\alpha}(z).
\end{equation*}

Given a parameter $b\in\mbb{C}$, the stress-energy tensor (Virasoro field) is defined by
\begin{equation*}
	T_{b}(z)=\frac{1}{2}\no{J(z)^{2}}+\sqrt{-1}b\del J(z),
\end{equation*}
and the corresponding central charge is checked to be $c^{\mrm{FB}}_{b}=1+12b^{2}$.
A vertex operator $V_{\alpha}(z)$, $\alpha\in\mbb{C}$ is a primary field of conformal weight $h^{\mrm{FB}}_{b}(\alpha)=\alpha (2b-\alpha)/2$
with respect to $T_{b}(z)$. In fact, it exhibits the following OPE with $T_{b}(z)$:
\begin{equation*}
	T_{b}(z)V_{\alpha}(w)\sim \frac{h^{\mrm{FB}}_{b}(\alpha)V_{\alpha}(w)}{(z-w)^{2}}+\frac{\del V_{\alpha}(w)}{z-w}.
\end{equation*}

For $\kappa>0$, we adopt the parametrization
\begin{equation*}
	b(\kappa)=\sqrt{\kappa/8}+\sqrt{2/\kappa},\quad \alpha_{+}(\kappa)=-\sqrt{2/\kappa},\quad \alpha_{-}(\kappa)=\sqrt{\kappa/2}+6/\sqrt{2\kappa}.
\end{equation*}
Then, we have $c^{\mrm{FB}}_{b(\kappa)}=c_{\kappa}=1+\frac{3(\kappa+4)^{2}}{2\kappa}$ and $h^{\mrm{FB}}_{b(\kappa)}(\alpha_{\pm}(\kappa))=h_{\kappa}$.

We also consider a Liouville CFT.
Let $\Psi_{h}$, $h\in\mbb{C}$ be a Virasoro primary field of conformal weight $h$ and set
\begin{equation*}
	\mcal{Z}(x_{1},\dots, x_{N})=\braket{h|\Psi_{h_{\kappa}}(x_{1})\cdots \Psi_{h_{\kappa}}(x_{N})|0},
\end{equation*}
where $\ket{0}$ is the vacuum vector of central charge $c_{\kappa}$, and $\bra{h}$ is the dual of a suitable highest weight vector
so that the above correlation function is non-trivial.
Since the field $\Psi_{h_{\kappa}}$ is degenerate, the correlation function $Z(x_{1},\dots, x_{N})$ satisfies BPZ equations:
\begin{equation}
\label{eq:BPZ}
	\mcal{D}^{\kappa}_{i}\mcal{Z}=0,\quad i=1,\dots, N.
\end{equation}
Therefore, the function $\mcal{Z}$ is considered as an $(N,\kappa)$-partition function.

Under the free boson theory, we set
\begin{align}
\label{eq:correlation_free_boson}
	\mcal{Z}^{\mrm{FB}}(x_{1},\dots,x_{N})
	&=\braket{N\alpha_{+}(\kappa)|V_{\alpha_{+}(\kappa)}(x_{1})\cdots V_{\alpha_{+}(\kappa)}(x_{N})|0} \\
	&=\prod_{1\le i<j\le N}(x_{i}-x_{j})^{-2/\kappa},\quad  x_{1}>x_{2}>\cdots >x_{N}, \notag
\end{align}
where $\ket{\alpha}$ is the vacuum vector of charge $\alpha$ and $\bra{\alpha}$ is its dual.
Then, the above correlation function $\mcal{Z}^{\mrm{FB}}(x_{1},\dots, x_{N})$ satisfies the system of BPZ equations (\ref{eq:BPZ}).

Next, we consider the {\it correlation function} 
\begin{equation*}
	\widetilde{\mcal{X}}(z,x_{1},\dots, x_{N})=\sqrt{-2}\braket{h|V_{\alpha}(z)\Psi_{h_{\kappa}}(x_{1})\cdots \Psi_{h_{\kappa}}(x_{N})|0},
\end{equation*}
which does not, however, make a rigorous representation theoretical sense because the vertex operator $V_{\alpha}(z)$ does not act on a state space of a Liouville CFT.
Nevertheless, the above description verifies a defining property of $\mcal{X}(z,x_{1},\dots,x_{N})$ in (\ref{eq:def_function_X})
as a solution of a system of differential equations.
Regarding the vertex operator $V_{\alpha}(z)$ as a primary field of conformal weight $h^{\mrm{FB}}_{b(\kappa)}(\alpha)$, we see that
\begin{equation}
\label{eq:BPZ_vertex}
	\widetilde{\mcal{D}}_{z,i}^{\kappa,\alpha}\widetilde{\mcal{X}}=0,\quad i=1,\dots, N,
\end{equation}
where
\begin{align*}
	\widetilde{\mcal{D}}_{z,i}^{\kappa,\alpha}=&\frac{\kappa}{2}\del_{x_{i}}^{2}-2\sum_{j;j\neq i}\left(\frac{1}{x_{j}-x_{i}}\del_{x_{j}}-\frac{h_{\kappa}}{(x_{j}-x_{i})^{2}}\right)-\frac{2}{z-x_{i}}\del_{z}+\frac{2h^{\mrm{FB}}_{b(\kappa)}(\alpha)}{(z-x_{i})^{2}}, \\
	& i=1,\dots, N.
\end{align*}
Applying the directional derivative $-\sqrt{-1}\frac{d}{d\alpha}|_{\alpha=0}$ to (\ref{eq:BPZ_vertex}), we see that the {\it correlation function}
\begin{equation*}
	\mcal{X}(z,x_{1},\dots, x_{N})=\sqrt{-2}\braket{h|\phi(z)\Psi_{h_{\kappa}}(x_{1})\cdots \Psi_{h_{\kappa}}(x_{N})|0}	
\end{equation*}
satisfies the system of differential equations
\begin{equation*}
	(\mcal{D}_{z,i}^{\kappa}\mcal{X})(z,x_{1},\dots, x_{N})+\frac{2Q}{(z-x_{i})^{2}}\mcal{Z}(x_{1},\dots,x_{N})=0,\quad i=1,\dots, N,
\end{equation*}
where $Q=\frac{2}{\sqrt{\kappa}}+\frac{\sqrt{\kappa}}{2}$.
Therefore, the function $\mcal{X}$ here is identified the function in (\ref{eq:def_function_X}).

We also remark that the correlation function $\mcal{Z}^{\mrm{FB}}$ in (\ref{eq:correlation_free_boson})
satisfies an additional system of differential equations.
Noticing the property
\begin{equation*}
	\del_{z}V_{\alpha}(z)=\sqrt{-1}\alpha\no{J(z)V_{\alpha}(z)}
\end{equation*}
and an OPE
\begin{equation*}
	J(z)V_{\alpha}(w)\sim \frac{1}{z-w}\no{J(w)V_{\alpha}(w)},
\end{equation*}
we see that
\begin{equation*}
	(\del_{x_{1}}\mcal{Z}^{\mrm{FB}})(x_{1},\dots, x_{N})=\sum_{j;j\neq i}\frac{-\alpha_{+}(\kappa)^{2}}{x_{i}-x_{j}}\mcal{Z}(x_{1},\dots, x_{N}),\quad i=1,\dots, N.
\end{equation*}
These are exactly the same as (\ref{eq:additional_condition}) and are regarded as {\it Knizhnik--Zamolodchikov equations}.

Let $(f_{t}(\cdot):t\ge 0)$ be a backward SLE$(\kappa)$
and write $W_{t}=\sqrt{\kappa}B_{t}$, $t\ge 0$ with $(B_{t}:t\ge 0)$ being a standard Brownian motion for its driving process.
The group theoretical formulation of SLE \cite{BauerBernard2003, BauerBernard2004}
(see also \cite[Appendix B]{KatoriKoshida2020a} and \cite[Section II]{SK2018}) associates to it an operator valued stochastic process
$(R(f_{t}):t\ge 0)$ satisfying
\begin{equation*}
	\frac{d}{dt}R(f_{t})=2R(f_{t})e^{L_{-1}W_{t}}L_{-2}e^{-L_{-1}W_{t}},\quad t\ge 0,\quad R(f_{0})=\mrm{Id},
\end{equation*}
where $L_{n}$, $n\in \mbb{Z}$ are the standard generators of the Virasoro algebra.
A primary field $\Psi_{h}$ behaves under conjugation by $R(f_{t})$, $t\ge 0$ as
\begin{equation*}
	R(f_{t})^{-1}\Psi_{h}(z)R(f_{t})=\pr{f_{t}}(z)^{h}\Psi(f_{t}(z)),\quad t\ge 0.	
\end{equation*}
Regarding a vertex operator $V_{\alpha}(z)$ as a primary field of conformal weight $h^{\mrm{FB}}_{b}(\alpha)$, we see that
it behaves in the same manner.
Then, the application of the directional derivative $-\sqrt{-1}\frac{d}{d\alpha}|_{\alpha=0}$ leads to
\begin{equation*}
	R(f_{t})^{-1}\phi(z) R(f_{t})=\phi (f_{t}(z))-\frac{\sqrt{-1}b}{2}\log \pr{f}_{t}(z),\quad t\ge 0.	
\end{equation*}

Owing to the fact that $(2L_{-2}+\frac{\kappa}{2}L_{-1}^{2})\ket{h_{\kappa}}=0$ in the irreducible representation of central charge $c_{\kappa}$, 
and the property $\Psi_{h_{\kappa}}(W_{t})\ket{0}=e^{L_{-1}W_{t}}\ket{h_{\kappa}}$,
the vector valued stochastic process
\begin{equation*}
	R(f_{t})\Psi_{h_{\kappa}}(W_{t})\ket{0},\quad t\ge 0
\end{equation*}
is a local martingale.
Therefore, it follows that, for $z\in\mbb{H}$, $\bm{X}\in\mrm{Conf}_{N}(\mbb{R})$ and $i\in\{1,\dots, N\}$, the stochastic processes
\begin{align*}
	M^{(i)}_{\bm{X},t}:&=\Braket{h|\Psi_{h_{\kappa}}(X_{1})\cdots \what{\Psi_{h_{\kappa}}(X_{i})}\cdots \Psi_{h_{\kappa}}(X_{N})R(f_{t})\Psi_{h_{\kappa}}(W_{t})|0} \\
	&=\prod_{j;j\neq i}\pr{f}_{t}(X_{j})^{h_{\kappa}}\mcal{Z}\left(f_{t}(X_{1}),\dots,\overset{i}{\check{W}_{t}},\dots, f_{t}(X_{N})\right),\quad t\ge 0
\end{align*}
and
\begin{align*}
	N^{(i)}_{\bm{X},t}(z):&= \sqrt{-2}\Braket{h|\phi(z)\Psi_{h_{\kappa}}(X_{1})\cdots \what{\Psi_{h_{\kappa}}(X_{i})}\cdots \Psi_{h_{\kappa}}(X_{N})R(f_{t})\Psi_{h_{\kappa}}(W_{t})|0} \\
	&= \prod_{j;j\neq i}\pr{f}_{t}(X_{j})^{h_{\kappa}}\mcal{X}\left(f_{t}(z),f_{t}(X_{1}),\dots,\overset{i}{\check{W}_{t}},\dots, f_{t}(X_{N})\right) \\
	&\hspace{20pt} +Q\log\pr{f}_{t}(z)\prod_{j;j\neq i}\pr{f}_{t}(X_{j})^{h_{\kappa}}\mcal{Z}\left(f_{t}(X_{1}),\dots,\overset{i}{\check{W}_{t}},\dots, f_{t}(X_{N})\right),\quad t\ge 0	
\end{align*}
are the local martingales that appeared in (\ref{eq:martingale_M}) and (\ref{eq:martingale_N}), respectively.

\section{Forward flow case}
\label{sect:forward_flow}
The aim of this appendix is to present an analogue of Theorem \ref{thm:coupling_constraint} in the case of forward flow.
The coupling between forward SLEs and GFFs has already been studied in many places~\cite{Dubedat2009,SchrammSheffield2013,IzyurovKytola2013, MillerSheffield2016a,PeltolaWu2019} (see also \cite{KatoriKoshida2020a,KatoriKoshida2020b}).
In these literatures, it has been shown that certain variants of SLE that include members of commuting Loewner chains are coupled with GFFs under specific boundary conditions.
They did not, however, excluded the possibility that other multiple SLEs are coupled with GFFs under other boundary conditions.
We will exclude this possibility below.

To make notations simpler, we use the same symbols as in the main text with different definition.
Therefore, readers are recommended to read this appendix separately from the main text.
At the same time, we give all descriptions in detail so that readers do not need to refer to the main text to read this appendix.

\subsection{Multiple SLE}
We define a multiple SLE as a multiple of probability measures.
Let $(B_{t}:t\ge 0)$ be a Brownian motion and write its law as $\mbb{P}$.
The law of a Brownian motion starting at $x$ will be denoted by $\mbb{P}^{x}$.
For a parameter $\kappa>0$, we consider an SLE$(\kappa)$ \cite{Schramm2000}, which is a Loewner chain $(g_{t}(\cdot):t\ge 0)$ satisfying
\begin{equation*}
	\frac{d}{dt}g_{t}(z)=\frac{2}{g_{t}(z)-W_{t}},\quad W_{t}=\sqrt{\kappa}B_{t},\quad t\ge 0,\quad g_{0}(z)=z\in\mbb{H}.
\end{equation*}
If we set $\eta (t):=\lim_{\epsilon\downarrow 0}g^{-1}_{t}(W_{t}+\sqrt{-1}\epsilon)$, $t\ge 0$,
then $\eta:[0,\infty)\to\overline{\mbb{H}}$ is almost surely a continuous curve \cite{RohdeSchramm2005},
which we call an SLE$(\kappa)$-curve.
Also we write $\mbb{H}_{t}$ for the unbounded component of $\mbb{H}\backslash \eta (0,t]$, $t\ge 0$
and set $K_{t}:=\mbb{H}\backslash\mbb{H}_{t}$.
Then,
\begin{equation*}
	g_{t}:\mbb{H}_{t}:=\mbb{H}\backslash K_{t}\to\mbb{H}
\end{equation*}
is the hydrodynamically normalized conformal equivalence at each $t\ge 0$.

For $N\in\mbb{N}$ and $\kappa>0$, an $(N,\kappa)$-partition function $\mcal{Z}$ is a translation invariant homogeneous function
on $\mrm{Conf}_{N}(\mbb{R})$ such that $\mcal{D}_{i}^{\kappa}\mcal{Z}=0$, $i=1,\dots, N$, where
\begin{equation*}
	\mcal{D}_{i}^{\kappa}=\frac{\kappa}{2}\del_{x_{i}}^{2}+2\sum_{j;j\neq i}\left(\frac{1}{x_{j}-x_{i}}\del_{x_{j}}-\frac{h_{\kappa}}{(x_{j}-x_{i})^{2}}\right),\quad i=1,\dots, N
\end{equation*}
with $h_{\kappa}=\frac{6-\kappa}{2\kappa}$.
We also assume that an $(N,\kappa)$-partition function is analytic in $\mrm{Conf}_{N}(\mbb{R})$ and admits the Frobenius expansion.
Solutions to this system of differential equations are studied in detail
in \cite{FloresKleban2015a,FloresKleban2015b,FloresKleban2015c,FloresKleban2015d,KytolaPeltola2016,PeltolaWu2019,KytolaPeltola2020}.
Usually, given an $(N,\kappa)$-partition function, the corresponding multiple SLE is defined as a multiple of Loewner chains properly constructed
\cite{KytolaPeltola2016,PeltolaWu2019}.
In this appendix, however, we directly construct  Girsanov transforms to define a multiple SLE.

Let $\left(g_{t}(\cdot):t\ge 0\right)$ be an SLE$(\kappa)$ driven by $(W_{t}:t\ge 0)$
and let $\mcal{Z}$ be an $(N,\kappa)$-partition function.
For $\bm{X}=(X_{1},\dots, X_{N})\in\mrm{Conf}_{N}(\mbb{R})$ and $i\in\{1,\dots, N\}$, we consider the stochastic process
$\left(M^{(i)}_{\bm{X},t}:t\ge 0\right)$ defined by
\begin{align*}
	M^{(i)}_{\bm{X},t}&=\prod_{j;j\neq i}\pr{g}_{t}(X_{j})^{h_{\kappa}}\mcal{Z}\left(X^{(1)}_{t},\dots,\overset{i}{\check{W}_{t}},\dots, X^{(N)}_{t}\right),\quad t\ge 0,\\
	\frac{d}{dt}X^{(j)}_{t}&=\frac{2}{X^{(j)}_{t}-W_{t}},\quad t\ge 0,\quad j\neq i
\end{align*}
under the probability measure $\mbb{P}^{X_{i}}$.
Here, $\pr{g}_{t}(z)$ is the derivative in terms of $z$.
For each $i\in\{1,\dots, N\}$, $n\in\mbb{N}$ and $\bm{X}\in\mrm{Conf}_{N}(\mbb{R})$, we set
\begin{equation*}
	\tau^{(i)}_{\bm{X},n}:=\mrm{inf}\left\{t>0\Big|\left|M^{(i)}_{\bm{X},t}\right|>n\right\}.
\end{equation*}
It it checked that it is a local martingale with increment
\begin{equation*}
	dM^{(i)}_{\bm{X},t}=s^{(i)}_{\bm{X},t}M^{(i)}_{\bm{X},t}dB_{t},\ \ s^{(i)}_{\bm{X},t}=\sqrt{\kappa}(\del_{x_{i}}\log\mcal{Z})\left(X^{(1)}_{t},\dots,\overset{i}{\check{W}_{t}},\dots, X^{(N)}_{t}\right),\quad t\ge 0.
\end{equation*}
We define the stochastic process $\left(B^{(i)}_{n,t}:t\ge 0\right)$ by
\begin{equation}
\label{eq:transformed_BM_forward}
	B^{(i)}_{n,t}=B_{t}-\int_{0}^{t\wedge \tau^{(i)}_{\bm{X},n}}s^{(i)}_{\bm{X},s}ds,\quad t\ge 0.
\end{equation}
Then, by Girsanov--Maruyama's theorem, this is a Brownian motion under the probability measure $\mbb{Q}^{(i)}_{\bm{X},n}$ defined by
\begin{equation}
\label{eq:transformed_measure_forward}
	\frac{d\mbb{Q}^{(i)}_{\bm{X},n}}{d\mbb{P}^{X_{i}}}=\lim_{t\to\infty} \frac{M^{(i)}_{\bm{X},t\wedge\tau^{(i)}_{\bm{X},n}}}{M^{(i)}_{\bm{X},0}}.
\end{equation}

\begin{defn}
Let $N\in\mbb{N}$ and $\kappa>0$.
Take an $(N,\kappa)$-partition function $\mcal{Z}$ and $\bm{X}\in\mrm{Conf}_{N}(\mbb{R})$.
A $\mcal{Z}$-multiple SLE$(\kappa)$ starting at $\bm{X}$ is a family of probability measures
$\left\{\mbb{Q}^{(i)}_{\bm{X},n}:i=1,\dots, N, n\in\mbb{N}\right\}$, each of which is defined by (\ref{eq:transformed_measure_forward}).
\end{defn}

It has been shown \cite{PeltolaWu2019} that this construction of multiple SLE coincides with a global definition of multiple SLE \cite{KozdronLawler2007,Lawler2009c,BeffaraPeltolaWu2018}.

\subsection{Dirichlet boundary GFF}
\label{subsect:Dirichlet_boundary_GFF}
Let $D\subsetneq \mbb{C}$ be a simply connected domain
and write $C^{\infty}_{0}(D)$ for the space of smooth functions on $D$ that are supported compactly.
We equip it with the Dirichlet inner product
\begin{equation*}
	(f,g)_{\nabla}=\frac{1}{2\pi}\int_{D} \nabla f\cdot \nabla g,\quad f,g\in C^{\infty}_{0}(D)
\end{equation*}
and write its Hilbert space completion as $W(D)$.

A Dirichlet boundary GFF on $D$ \cite{Sheffield2007} is a collection $\{(H,f)_{\nabla}|f\in W(D)\}$ of centered Gaussian random variables so that
\begin{equation*}
	\mbb{E}[(H,f)_{\nabla}(H,g)_{\nabla}]=(f,g)_{\nabla},\quad f,g\in W(D).
\end{equation*}
We write $\mcal{P}$ for the probability law of these Gaussian random variables.
Using the Dirichlet boundary Laplacian $\Delta$, we also set $(H,f):=2\pi (H,(-\Delta)^{-1}f)_{\nabla}$, $f\in W(D)$.
Then, we have
\begin{equation*}
	\mbb{E}[(H,f)(H,g)]=\int_{D\times D}f(z)G(z,w)g(w)dzdw,\quad f,g\in W(D),
\end{equation*}
where $G(z,w)$ is Dirichlet boundary Green's function of $D$.
It is reasonable that we formally write
\begin{equation*}
	(H,f)=\int_{D}H(z)f(z)dz,\quad f\in W(D)
\end{equation*}
and also call the random distribution $H$ a Dirichlet boundary GFF on $D$.
The desired covariance structure can be recovered by thinking of
\begin{equation*}
	\mbb{E}[H(z)H(w)]=G(z,w),\quad z,w\in D,\quad z\neq w.
\end{equation*}

\begin{exam}
In the case of $D=\mbb{H}$,
\begin{equation*}
	G_{\mbb{H}}(z,w)=-\log |z-w|+\log|z-\overline{w}|,\quad z,w\in \mbb{H},\quad z\neq w
\end{equation*}
is Drichlet boundary Green's function.
\end{exam}

\subsection{SLE/GFF-coupling}
\begin{defn}
A function $u=u(z;x_{1},\dots, x_{N})$ of $z\in\mbb{H}$ and $(x_{1},\dots, x_{N})\in \mrm{Conf}_{N}(\mbb{R})$
is called a boundary perturbation for Dirichlet boundary GFF if it is harmonic in $z\in\mbb{H}$
and has the following properties.
\begin{description}
\item[Translation invariance] 
	For any $a\in\mbb{R}$, we have
	\begin{equation*}
		u(z+a;x_{1}+a,\dots, x_{N}+a)=u(z;x_{1},\dots, x_{N}),\quad z\in\mbb{H},\quad (x_{1},\dots, x_{N})\in\mrm{Conf}_{N}(\mbb{R}).
	\end{equation*}
\item[Scale invariance]
	For any $\lambda>0$, we have
	\begin{equation*}
		u(\lambda z;\lambda x_{1},\dots, \lambda x_{N})=u(z;x_{1},\dots, x_{N}),\quad z\in\mbb{H},\quad (x_{1},\dots, x_{N})\in \mrm{Conf}_{N}(\mbb{R}).
	\end{equation*}
\end{description}
\end{defn}

For a boundary perturbation $u=u(z;x_{1},\dots, x_{N})$ and $\bm{X}\in\mrm{Conf}_{N}(\mbb{R})$,
we call the random distribution $H_{(u,\bm{X})}:=H+u(\cdot;\bm{X})$ with $H$ being a Dirichlet boundary GFF
a $(u,\bm{X})$-perturbed Dirichlet boundary GFF.
Note that for $f\in C^{\infty}_{0}(\mbb{H})$, which is compactly supported, we have
$(H_{(u,\bm{X})},f)=(H,f)$ a.s.
That is, a $(u,\bm{X})$-perturbed Dirichlet boundary GFF cannot be distinguished from the original Dirichlet boundary GFF
by a test function supported in the {\it bulk}.

Given a boundary perturbation $u=u(z;x_{1},\dots, x_{N})$ and $\bm{X}\in\mrm{Conf}_{N}(\mbb{R})$,
for each $i\in\{1,\dots, N\}$, we consider the following stochastic process
\begin{align*}
	\mfrak{h}^{(i)}_{\bm{X},t}(z)&=u\left(g_{t}(z);X^{(1)}_{t},\dots, \overset{i}{\check{W}_{t}},\dots, X^{(N)}_{t}\right)-\chi \arg \pr{g}_{t}(z),\quad t\ge 0,\\
	\frac{d}{dt}X^{(j)}_{t}&=\frac{2}{X^{(j)}_{t}-W_{t}},\quad t\ge 0,\quad X^{(j)}_{0}=X_{j},\quad j\neq i
\end{align*}
under $\mbb{P}^{X_{i}}$,
where $(g_{t}(\cdot):t\ge 0)$ is an SLE$(\kappa)$ driven by $(W_{t}:t\ge 0)$ and $\chi >0$.
We also assume that the probability measure $\mbb{P}^{X_{i}}$ is independent of the law of a Dirichlet boundary GFF.

\begin{defn}
Let $N\in\mbb{N}$, $\kappa>0$ and $\mcal{Z}$ be an $(N,\kappa)$-partition function.
We also let $u=u(z;x_{1},\dots, x_{N})$ be a boundary perturbation for Dirichlet boundary GFF.
For $\bm{X}\in\mrm{Conf}_{N}(\mbb{R})$, we say that a $\mcal{Z}$-multiple SLE$(\kappa)$ starting at $\bm{X}$
is coupled with a $(u,\bm{X})$-perturbed Dirichlet boundary GFF $H_{(u,\bm{X})}$ with coupling constant $\chi$ if,
for every $n\in\mbb{N}$, each $\left(\mfrak{h}^{(i)}_{\bm{X},t}:t\ge 0\right)$ is a $\mbb{Q}^{(i)}_{\bm{X},n}$-local martingale with
cross variation given by
\begin{equation*}
	d\left[\mfrak{h}^{(i)}_{\bm{X}}(z),\mfrak{h}^{(i)}_{\bm{X}}(w)\right]_{t}=-dG_{t}(z,w),\quad z,w\in\mbb{H},\quad t\ge 0,
\end{equation*}
where $G_{t}(z,w):=G_{\mbb{H}}(g_{t}(z),g_{t}(w))$, $z,w\in \mbb{H}_{t}$, $z\neq w$.
\end{defn}

To motivate this definition, let us consider the following stochastic processes.
For $\bm{X}\in\mrm{Conf}_{N}(\mbb{R})$ and each $i\in \{1,\dots, N\}$, set
\begin{equation*}
	\mfrak{p}^{(i)}_{\bm{X},t}:=\mfrak{h}^{(i)}_{\bm{X},t}+H\circ g_{t},\quad t\ge 0.
\end{equation*}
At $t=0$, we have $\mfrak{p}^{(i)}_{\bm{X},0}=H_{(u,\bm{X})}$ regardless of $i\in \{1,\dots, N\}$.

\begin{prop}
Suppose that a $\mcal{Z}$-multiple SLE$(\kappa)$ starting at $\bm{X}\in\mrm{Conf}_{N}(\mbb{R})$ is coupled with
a $(u,\bm{X})$-perturbed Dirichlet boundary GFF $H_{(u,\bm{X})}$ with parameter $\chi$.
Then, at each time $t\ge 0$, the law of $\mfrak{p}^{(i)}_{\bm{X},t}$ under $\mcal{P}\otimes \mbb{Q}^{(i)}_{\bm{X},n}$ is identical to that of $H_{(u,\bm{X})}$
under $\mcal{P}$ (see Subsect.~\ref{subsect:Dirichlet_boundary_GFF}) for every $i=1,\dots, N$ and $n\in\mbb{N}$.
\end{prop}
\begin{proof}
The proof is identical to the case of backward flow, but we present it here again.
It can be seen that
\begin{equation*}
	\mbb{E}\left[e^{\sqrt{-1}\zeta(H_{(u,\bm{X})},\rho)}\right]=e^{\sqrt{-1}\zeta (u(\cdot;\bm{X}),\rho)-\frac{\zeta^{2}}{2}E(\rho)},\quad \rho\in W(\mbb{H}),\quad \zeta\in\mbb{R},
\end{equation*}
where we set
\begin{equation*}
	E(\rho):=\int_{\mbb{H}\times\mbb{H}}\rho(z)G_{\mbb{H}}(z,w)\rho(w)dzdw
\end{equation*}
for the Dirichlet energy of $\rho\in W(\mbb{H})$.

On the other hand, writing $(\mcal{F}_{t})_{t\ge 0}$ for the filtration associated with a $\mbb{P}^{X_{i}}$-Brownian motion $(B_{t}:t\ge 0)$,
we have
\begin{align*}
	\mbb{E}\left[e^{\sqrt{-1}\zeta(\mfrak{p}^{(i)}_{\bm{X},t},\rho)}\right]
	&=\mbb{E}\left[e^{\sqrt{-1}\zeta (\mfrak{h}^{(i)}_{\bm{X},t},\rho)}\mbb{E}\left[e^{\sqrt{-1}\zeta (H\circ g_{t},\rho)}\Big|\mcal{F}_{t}\right]\right] \\
	&=\mbb{E}\left[e^{\sqrt{-1}\zeta (\mfrak{h}^{(i)}_{\bm{X},t},\rho)-\frac{\zeta^{2}}{2}E_{t}(\rho)}\right],\quad \rho\in W(\mbb{H}),\quad \zeta\in\mbb{R},
\end{align*}
where we set
\begin{equation*}
	E_{t}(\rho)=\int_{\mbb{H}_{t}\times\mbb{H}_{t}}\rho(z)G_{t}(z,w)\rho(w)dzdw,\quad \rho\in W(\mbb{H}).
\end{equation*}
Here, we restrict the test function on $\mbb{H}_{t}$.
By assumption, we have $d\left[(\mfrak{h}^{(i)}_{\bm{X}},\rho)\right]_{t}=-dE_{t}(\rho)$, $t\ge 0$, which ensures that
$\left[(\mfrak{h}^{(i)}_{\bm{X}},\rho)\right]_{t}=-E_{t}(\rho)+E(\rho)$, $t\ge 0$.
This leads to
\begin{align*}
	\mbb{E}\left[e^{\sqrt{-1}\zeta(\mfrak{p}^{(i)}_{\bm{X},t},\rho)}\right]
	&=e^{-\frac{\zeta^{2}}{2}E(\rho)}\mbb{E}\left[e^{\sqrt{-1}\zeta (\mfrak{h}^{(i)}_{\bm{X},t},\rho)+\frac{\zeta^{2}}{2}[(\mfrak{h}^{(i)}_{\bm{X}},\rho)]_{t}}\right],
\end{align*}
where $\left(e^{\sqrt{-1}\zeta (\mfrak{h}^{(i)}_{\bm{X},t},\rho)+\frac{\zeta^{2}}{2}[(\mfrak{h}^{(i)}_{\bm{X}},\rho)]_{t}}:t\ge 0\right)$ is a martingale.
Therefore, we have
\begin{equation*}
	\mbb{E}\left[e^{\sqrt{-1}\zeta(\mfrak{p}^{(i)}_{\bm{X},t},\rho)}\right]=e^{\sqrt{-1}\zeta (u(\cdot;\bm{X}),\rho)-\frac{\zeta^{2}}{2}E(\rho)}
	=\mbb{E}\left[e^{\sqrt{-1}\zeta (H_{(u,\bm{X})},\rho)}\right],\quad \zeta\in\mbb{R},
\end{equation*}
which gives the desired result.
\end{proof}
This proposition admits an interpretation in terms of the flow line problem \cite{Sheffield2016,MillerSheffield2016a,KatoriKoshida2020a}.
Indeed, it says that the $i$-th curve is the flow line starting at $X_{i}$ 
along a random vector field generated by $H_{(u,\bm{X})}$.

The main result here is the following theorem.
\begin{thm}
\label{thm:coupling_forward}
Let $N\in\mbb{N}$, $0<\kappa\neq 4$ and $\mcal{Z}$ be an $(N,\kappa)$-partition function.
We also let $u=u(z;x_{1},\dots, x_{N})$ be a boundary perturbation for Dirichlet boundary GFF.
A $\mcal{Z}$-multiple SLE$(\kappa)$ is coupled with a $(u,\bm{X})$-perturbed Dirichlet boundary GFF $H_{(u,\bm{X})}$ with coupling constant $\chi > 0$
for arbitrary $\bm{X}\in\mrm{Conf}_{N}(\mbb{R})$ if and only if the following conditions are satisfied:
\begin{enumerate}
\item 	The $(N,\kappa)$-partition function is
		\begin{equation*}
			\mcal{Z}(x_{1},\dots,x_{N})=\prod_{1\le i<j\le N}|x_{i}-x_{j}|^{2/\kappa}
		\end{equation*}
		up to multiplication by nonzero constants.
\item 	Either
		\begin{enumerate}
		\item 	The parameters are related as $\chi=\frac{2}{\sqrt{\kappa}}-\frac{\sqrt{\kappa}}{2}$, $0<\kappa<4$.
		\item 	The boundary perturbation is given by
				\begin{equation*}
					u(z;x_{1},\dots, x_{N})=-\frac{2}{\sqrt{\kappa}}\sum_{i=1}^{N}\arg (z-x_{i})
				\end{equation*}
				up to addition of constants.
		\end{enumerate}
		or
		\begin{enumerate}
		\item 	The parameters are related as $\chi=-\frac{2}{\sqrt{\kappa}}+\frac{\sqrt{\kappa}}{2}$, $\kappa>4$.
		\item 	The boundary perturbation is given by
				\begin{equation*}
					u(z;x_{1},\dots, x_{N})=\frac{2}{\sqrt{\kappa}}\sum_{i=1}^{N}\arg (z-x_{i})
				\end{equation*}
				up to addition of constants.
		\end{enumerate}
		holds.
\end{enumerate}
\end{thm}

Before proving Theorem \ref{thm:coupling_forward}, we note the following fact.
\begin{lem}
\label{lem:coupling_condition_forward}
A $\mcal{Z}$-multiple SLE$(\kappa)$ is coupled with $H_{(u,\bm{X})}$ with coupling constant $\chi$
if and only if there exists a sequence $\bm{\epsilon}=(\epsilon_{i}\in\{\pm 1\})$ such that
the increment of $\left(\mfrak{h}^{(i)}_{\bm{X},t}:t\ge 0\right)$ becomes
\begin{equation*}
	d\mfrak{h}^{(i)}_{\bm{X},t}(z)=\mrm{Im}\frac{2\epsilon_{i}}{g_{t}(z)-W_{t}}dB^{(i)}_{n,t},\ \ z\in\mbb{H},\ \ t\ge 0
\end{equation*}
for every $i\in\{1,\dots, N\}$ and $n\in\mbb{N}$,
where $(B^{(i)}_{n,t}:t\ge 0)$ is a $\mbb{Q}^{(i)}_{\bm{X},n}$-Brownian motion defined by (\ref{eq:transformed_BM_forward}).
\end{lem}
\begin{proof}
Note that we have
\begin{equation*}
	G_{t}(z,w)=-\log |g_{t}(z)-g_{t}(w)|+\log |g_{t}(z)-\overline{g_{t}(w)}|,\quad z,w\in\mbb{H}_{t},\quad z\neq w,\quad t\ge 0.
\end{equation*}
The assertion immediately follows from the fact that
\begin{equation*}
	dG_{t}(z,w)=-\mrm{Im}\frac{2}{g_{t}(z)-W_{t}}\mrm{Im}\frac{2}{g_{t}(w)-W_{t}}dt,\quad z,w\in\mbb{H}_{t},\quad t\ge 0
\end{equation*}
holds.
\end{proof}

\begin{proof}[Proof of Theorem \ref{thm:coupling_forward}]
Let $\tilde{u}=\tilde{u}(z;x_{1},\dots, x_{N})$, $z\in\mbb{H}$, $(x_{1},\dots, x_{N})\in\mrm{Conf}_{N}(\mbb{R})$ be a holomorphic function in $z$ so that
\begin{equation*}
	u(z;x_{1},\dots, x_{N})=\mrm{Im}\tilde{u}(z;x_{1},\dots, x_{N}),\quad z\in\mbb{H},\quad (x_{1},\dots, x_{N})\in\mrm{Conf}_{N}(\mbb{R}).
\end{equation*}
Such a function is determined uniquely up to addition of constants.
Then, for $\bm{X}\in\mrm{Conf}_{N}(\mbb{R})$ and $i\in\{1,\dots, N\}$, the stochastic process $\left(\mfrak{h}^{(i)}_{\bm{X},t}:t\ge 0\right)$
is the imaginary part of
\begin{equation*}
	\tilde{\mfrak{h}}^{(i)}_{\bm{X},t}(z)=\tilde{u}\left(g_{t}(z);X^{(1)}_{t},\dots,\overset{i}{\check{W}_{t}},\dots, X^{(N)}_{t}\right)-\chi \log \pr{g}_{t}(z),\quad z\in\mbb{H},\\quad t\ge 0.
\end{equation*}
We set $\left(N^{(i)}_{\bm{X},t}=\tilde{\mfrak{h}}^{(i)}_{\bm{X},t}M^{(i)}_{\bm{X},t}:t\ge 0\right)$, $\bm{X}\in\mrm{Conf}_{N}(\mbb{R})$, $i\in\{1,\dots, N\}$.
Then, the stochastic process $\left(\tilde{\mfrak{h}}^{(i)}_{\bm{X},t}:t\ge 0\right)$ is a $\mbb{Q}^{(i)}_{\bm{X},n}$-local martingale
if and only if $\left(N^{(i)}_{\bm{X},t}:t\ge 0\right)$ is a $\mbb{P}^{X_{i}}$-local martingale.
For convenience, we set
\begin{equation*}
	\mcal{X}(z;x_{1},\dots, x_{N})=\tilde{u}(z;x_{1},\dots, x_{N})\mcal{Z}(x_{1},\dots, x_{N}),\quad z\in\mbb{H},\quad (x_{1},\dots, x_{N})\in\mrm{Conf}_{N}(\mbb{R}).
\end{equation*}
By direct computation, the increment of $\left(N^{(i)}_{\bm{X},t}:t\ge 0\right)$ is given by
\begin{align*}
	dN^{(i)}_{\bm{X},t}(z)
	=&\prod_{j;j\neq i}\pr{g}_{t}(X_{j})^{h_{\kappa}}\Biggl[(\mcal{D}_{z,i}^{\kappa}\mcal{X})\left(g_{t}(z);X^{(1)}_{t},\dots, \overset{i}{\check{W}_{t}},\dots, X^{(N)}_{t}\right)\\
	&\hspace{70pt}+\frac{2\chi}{(g_{t}(z)-W_{t})^{2}}\mcal{Z}\left(X^{(1)}_{t},\dots,\overset{i}{\check{W}_{t}},\dots, X^{(N)}_{t}\right)\Biggr]dt \\
	&+\Biggl[\sqrt{\kappa}\prod_{j;j\neq i}\pr{g}_{t}(X_{j})^{h_{\kappa}}(\del_{x_{j}}\mcal{X})\left(g_{t}(z),X^{(1)}_{t},\dots, \overset{i}{\check{W}_{t}},\dots, X^{(N)}_{t}\right)\\
	&\hspace{20pt}-\chi\log \pr{g}_{t}(z)s^{(i)}_{\bm{X},t}M^{(i)}_{\bm{X},t}\Biggr] dB_{t},\quad t\ge 0,\quad z\in\mbb{H},
\end{align*}
where
\begin{equation*}
	\mcal{D}_{z,i}^{\kappa}:=\frac{\kappa}{2}\del_{x_{i}}^{2}+2\sum_{j;j\neq i}\left(\frac{1}{x_{j}-x_{i}}\del_{x_{j}}-\frac{h_{\kappa}}{(x_{j}-x_{i})^{2}}\right)+\frac{2}{z-x_{i}}\del_{z},\quad i=1,\dots, N.
\end{equation*}
Requiring that $\left(N^{(i)}_{\bm{X},t}:t\ge 0\right)$ is a $\mbb{P}^{X_{i}}$-local martingale for every $i\in \{1,\dots, N\}$
and an arbitrary initial condition $\bm{X}\in\mrm{Conf}_{N}(\mbb{R})$,
we see that the differential equations
\begin{align}
\label{eq:differential_equation_X_forward}
	(\mcal{D}^{\kappa}_{z,i}\mcal{X})(z;x_{1},\dots, x_{N})+\frac{2\chi}{(z-x_{i})^{2}}\mcal{Z}(x_{1},\dots, x_{N})=0,& \\
	z\in\mbb{H},\quad (x_{1},\dots, x_{N})\in\mrm{Conf}_{N}(\mbb{R}),\quad i=1,\dots, N& \notag
\end{align}
have to be satisfied.

Assuming (\ref{eq:differential_equation_X_forward}), we compute the increment of $\left(\tilde{\mfrak{h}}^{(i)}_{\bm{X},t}:t\ge 0\right)$ to obtain
\begin{equation*}
	d\tilde{\mfrak{h}}_{\bm{X},t}(z)=\sqrt{\kappa}(\del_{x_{i}}\tilde{u})\left(g_{t}(z);X^{(1)}_{t},\dots,\overset{i}{\check{W}_{t}},\dots, X^{(N)}_{t}\right)dB^{(i)}_{n,t},\quad z\in\mbb{H}_{t},\quad t\ge 0,
\end{equation*}
where $(B^{(i)}_{n,t}:t\ge 0)$ is a $\mbb{Q}^{(i)}_{\bm{X},n}$-Brownian motion defined by (\ref{eq:transformed_BM_forward}).
By Lemma \ref{lem:coupling_condition_forward}, there exists a sequence $\bm{\epsilon}=(\epsilon_{i}\in\{\pm 1\})$ so that we can require
\begin{equation*}
	(\del_{x_{i}}\tilde{u})(z;x_{1},\dots, x_{N})=\frac{2\epsilon_{i}/\sqrt{\kappa}}{z-x_{i}},\quad i=1,\dots, N.
\end{equation*}
They are solved by
\begin{equation*}
	\tilde{u}(z;x_{1},\dots, x_{N})=-\frac{2}{\sqrt{\kappa}}\sum_{i=1}^{N}\epsilon_{i}\log (z-x_{i})+h(z)
\end{equation*}
with $h(z)$ being a holomorphic function only of $z$.
For $u=\mrm{Im}\tilde{u}$ to be translation invariant, $h(z)$ must be a constant.

We again require $\mcal{X}=\tilde{u}\mcal{Z}$ with $\tilde{u}$ given above to solve (\ref{eq:differential_equation_X_forward}).
We have
\begin{align*}
	(\mcal{D}^{\kappa}_{z,i}\mcal{X})(z;x_{1},\dots, x_{N})
	&=\left(\frac{(\sqrt{\kappa}-4/\sqrt{\kappa})\epsilon_{i}}{(z-x_{i})^{2}}-\frac{4/\sqrt{\kappa}}{z-x_{i}}\sum_{j;j\neq i}\frac{\epsilon_{j}}{x_{i}-x_{j}}\right)\mcal{Z}(x_{1},\dots, x_{N}) \\
	&\hspace{20pt}+\frac{2\sqrt{\kappa}\epsilon_{i}}{z-x_{i}}(\del_{x_{i}}\mcal{Z})(x_{1},\dots, x_{N}).
\end{align*}
Therefore, either of the followings has to occur:
\begin{enumerate}
\item 	$\chi=\frac{2}{\sqrt{\kappa}}-\frac{\sqrt{\kappa}}{2}$ with $0<\kappa<4$ and $\epsilon_{i}=1$, $i=1,\dots, N$.
		In this case, we also have
		\begin{equation*}
			\tilde{u}(z;x_{1},\dots, x_{N})=-\frac{2}{\sqrt{\kappa}}\sum_{i=1}^{N}\log (z-x_{i})
		\end{equation*}
		up to additive constants.
\item 	$\chi=-\frac{2}{\sqrt{\kappa}}+\frac{\sqrt{\kappa}}{2}$ with $\kappa>4$ and $\epsilon_{i}=-1$, $i=1,\dots, N$.
		In this case, we also have
		\begin{equation*}
			\tilde{u}(z;x_{1},\dots, x_{N})=\frac{2}{\sqrt{\kappa}}\sum_{i=1}^{N}\log (z-x_{i})
		\end{equation*}
		up to additive constants.
\end{enumerate}
In both cases, the partition function $\mcal{Z}$ is subject to additional conditions
\begin{equation*}
	(\del_{x_{i}}\mcal{Z})(x_{1},\dots, x_{N})=\sum_{j;j\neq i}\frac{2/\kappa}{x_{i}-x_{j}}\mcal{Z}(x_{1},\dots, x_{N}),\quad i=1,\dots, N.
\end{equation*}
This implies that the partition function has asymptotic behavior
\begin{equation}
\label{eq:asymptitics_partition_function_forward}
	\mcal{Z}(x_{1},\dots, x_{N})\sim (x_{i}-x_{j})^{2/\kappa},\quad x_{i}\downarrow x_{j}
\end{equation}
for every pair $\{i,j\}\subset \{1,\dots, N\}$.
We can check that, in the asymptotic behavior of the $(N,\kappa)$-partition function
\begin{equation*}
	\mcal{Z}(x_{1},\dots, x_{N})\sim (x_{i}-x_{j})^{\Delta},\ \ x_{i}\downarrow x_{i},
\end{equation*}
the exponent $\Delta$ can be either $2/\kappa$ or $(6-\kappa)/\kappa$, which are distinct if $\kappa\neq 4$.
Therefore, following the general theory of partial differential equations with regular singular points \cite[Appendix B]{Knapp1986}, if $\kappa\neq 4$, the asymptotic behaviors (\ref{eq:asymptitics_partition_function_forward}) are
sufficient to fix the $(N,\kappa)$-partition function as
\begin{equation*}
	\mcal{Z}(x_{1},\dots, x_{N})=\prod_{1\le i<j\le N}|x_{i}-x_{j}|^{2/\kappa}
\end{equation*}
up to multiplicative constants,
which is certainly an $(N,\kappa)$-partition function.
\end{proof}

\begin{rem}
As we anticipated above, when $\kappa=4$, the requirement of asymptotic behaviors cannot fix a partition function because two possible exponents coincide.
Indeed, coupling with a multiple SLE$(4)$ and GFF was considered for any partition function \cite{PeltolaWu2019} in connection to the level lines of a GFF.
\end{rem}

\addcontentsline{toc}{chapter}{Bibliography}
\bibliographystyle{alpha}
\bibliography{sle_gff}

\end{document}